\documentclass[a4paper,12pt,final]{amsart}
\usepackage[a4paper,margin=30mm]{geometry}

\usepackage[T1]{fontenc}
\usepackage{amssymb,amsmath,amsthm,amscd}
\usepackage{eucal,mathrsfs,dsfont}
\usepackage{lscape,enumerate}

\usepackage{verbatim}
\usepackage[usenames,dvipsnames]{color}
\usepackage[colorlinks=true,urlcolor=blue,linkcolor=blue,citecolor=blue,pagebackref=false]{hyperref}

\newcommand{\comp}{\mathds C}

\newcommand{\nat}{\mathds N}

\newcommand{\real}{\mathds R}
\newcommand{\rn}{{{\mathds R}^n}}
\newcommand{\Rm}{{{\mathds R}^m}}
\newcommand{\rd}{{{\mathds R}^d}}
\newcommand{\Ee}{\mathds E}
\newcommand{\Pp}{\mathds P}

\newcommand{\Ff}{\mathcal{F}}

\newcommand{\mcn}{\mathcal{MCN}}
\newcommand{\CBF}{\mathcal{CBF}}
\newcommand{\BF}{\mathcal{BF}}

\newcommand{\Nsf}{\mathsf{N}}
\newcommand{\gauss}{\nu}

\newcommand{\nnorm}[1]{\|#1\|}

\newcommand{\scalar}[1]{\langle #1\rangle}

\renewcommand{\leq}{\leqslant}
\renewcommand{\geq}{\geqslant}
\renewcommand{\Re}{\ensuremath{\operatorname{Re}}}

\newcommand{\supp}{\operatorname{\mathrm{supp}}}

\newtheorem{theorem}{Theorem}
\newtheorem{lemma}[theorem]{Lemma}
\newtheorem{proposition}[theorem]{Proposition}
\newtheorem{corollary}[theorem]{Corollary}
\newtheorem{conjecture}[theorem]{Conjecture}
\newtheorem{problem}[theorem]{Problem}
\theoremstyle{definition}
\newtheorem{remark}[theorem]{Remark}
\newtheorem{definition}[theorem]{Definition}
\newtheorem{example}[theorem]{Example}
\newtheorem*{ack}{Acknowledgement}
\newtheorem*{notation}{Notation}
\numberwithin{theorem}{section}
\numberwithin{equation}{section}

\usepackage{color}

\newcommand{\vica}{\color{blue}}
\newcommand{\rene}{\color{red}}
\newcommand{\normal}{\color{black}}

\begin{document}\allowdisplaybreaks
\title[L\'evy Processes and Geometry]{A geometric interpretation of the transition density of a symmetric L\'evy Process}

\dedicatory{Dedicated to Professor Chen, Mu-Fa and Professor Ma, Zhi-Ming}

\author{N.\ Jacob, V.\ Knopova, S.\ Landwehr \and R.L.\ Schilling}
\thanks{%
    \emph{N.\ Jacob}: Mathematics Department, Swansea University, Singleton Park, Swansea SA2 8PP, UK,
    \href{mailto:n.jacob@swansea.ac.uk}{n.jacob@swansea.ac.uk}}
\thanks{%
    \emph{V.\ Knopova}: V.M.Glushkov Institute of Cybernetics NAS of Ukraine, 03187, Kiev,
    Ukraine, \href{mailto:vicknopova@googlemail.com}{vicknopova@googlemail.com}}
\thanks{%
    \emph{S.\ Landwehr}: Heinrich Heine University D\"{u}sseldorf, German Diabetes Center at the Heinrich Heine University D\"{u}sseldorf, Leibniz Center for Diabetes Research, Institute of Biometrics and Epidemiology, Auf'm Hennekamp 65, 40225 D\"{u}sseldorf, Germany, \href{mailto:sandra.landwehr@ddz.uni-duesseldorf.de}{sandra.landwehr@ddz.uni-duesseldorf.de}}
\thanks{%
    \emph{R.L.\ Schilling}: Institut f\"ur Mathematische Stochastik, Technische
    Universit\"at Dresden, 01062 Dresden, Germany, \href{mailto:rene.schilling@tu-dresden.de}{rene.schilling@tu-dresden.de}}

\maketitle

\begin{abstract}
    We study for a class of symmetric L\'evy processes with state space $\rn$ the transition density $p_t(x)$ in terms of two one-parameter families of metrics, $(d_t)_{t>0}$ and $(\delta_t)_{t>0}$. The first family of metrics describes the diagonal term $p_t(0)$; it is induced by the characteristic exponent $\psi$ of the L\'evy process by $d_t(x,y)=\sqrt{t\psi(x-y)}$. The second and new family of metrics $\delta_t$ relates to $\sqrt{t\psi}$ through the formula
    $$
        \exp\left(-\delta_t^2(x,y)\right)
        = \Ff\left[\frac{e^{-t\psi}}{p_t(0)}\right](x-y)
    $$
    where $\Ff$ denotes the Fourier transform. Thus we obtain the following ``Gaussian'' representation of the transition density: $p_t(x)=p_t(0) e^{-\delta_t^2(x,0)}$ where $p_t(0)$ corresponds to a volume term related to $\sqrt{t\psi}$ and where an ``exponential'' decay is governed by $\delta_t^2$. This gives a complete and new geometric, intrinsic interpretation of $p_t(x)$.


\end{abstract}

\section{Introduction}\label{intro}
We start with a simple example. Let $(C_t)_{t\geq 0}$ be the one-dimensional Cauchy process. Its transition function $p_t$ has a density with respect to Lebesgue measure in $\real$ and we denote the density again by $p_t$,
\begin{equation}\label{intro-e01}
    p_t(x,y) = \frac{t}{\pi\,(t^2 + |x-y|^2)}.
\end{equation}
Since $(C_t)_{t\geq 0}$ has stationary and independent increments, $p_t(x,y)$ depends only on the increment $x-y$, i.e.
\begin{equation*}
    p_t(x) := p_t(x,0) = \frac{t}{\pi\,(t^2 + |x|^2)}.
\end{equation*}
Let us introduce two one-parameter families of metrics, $d_{C,t}(\cdot,\cdot)$, $t>0$, and $\delta_{C,t}(\cdot,\cdot)$, $t>0$, on $\real$ defined by
\begin{equation*}
    d_{C,t}(x,y) := \sqrt{t\,|x-y|}
\end{equation*}
and
\begin{equation*}
    \delta_{C,t}(x,y) := \sqrt{\ln\left[\frac{|x-y|^2+t^2}{t^2}\right]}.
\end{equation*}
With
\begin{equation*}
    B^{d_{C,t}}(0,1)
    := \big\{ x\in\real \::\: d_{C,t}(x,0)<1\big\}
    = \big\{ x\in\real \::\: |x|<1/t\big\}
\end{equation*}
we find
\begin{equation*}
    \lambda\Big(B^{d_{C,t}}(0,1)\Big) = \frac 2t
\end{equation*}
and, therefore,
\begin{equation*}
    p_t(0) = \frac 1{2\pi}\,\lambda\Big(B^{d_{C,t}}(0,1)\Big).
\end{equation*}
Thus, we find
\begin{equation}\label{intro-e08}
    p_t(x) = \frac 1{2\pi}\,\lambda\Big(B^{d_{C,t}}(0,1)\Big)\,e^{-\delta_{C,t}^2(x,0)}.
\end{equation}
Let us compare \eqref{intro-e08} with the Gaussian, i.e.\ with the density $\gauss_t$ of the transition function of a one-dimensional Brownian motion $(B_t)_{t\geq 0}$ where
\begin{equation*}
    \gauss_t(x) = \frac 1{\sqrt{2\pi t}}\,e^{-|x|^2/2t}.
\end{equation*}
We may introduce the following two one-parameter families of metrics on $\real$
\begin{equation*}
    d_{G,t}(x,y) = \sqrt t\,|x-y|
\end{equation*}
and
\begin{equation*}
    \delta_{G,t}(x,y) = \frac 1{\sqrt{2t}}\,|x-y|.
\end{equation*}
For
$$
    B^{d_{G,t}}(0,1)
    := \big\{ x\in\real \::\: d_{G,t}(x,0)<1\big\}
    = \big\{ x\in\real \::\: |x|<1/\sqrt t\,\big\}
$$
it holds
\begin{equation*}
    \lambda\Big(B^{d_{G,t}}(0,1)\Big) = \frac 2{\sqrt t}
\end{equation*}
and, consequently,
\begin{equation*}
    \gauss_t(0) = \frac 1{\sqrt{8\pi}}\,\lambda\Big(B^{d_{G,t}}(0,1)\Big).
\end{equation*}
Hence, we find
\begin{equation}\label{intro-e14}
    \gauss_t(x) = \frac 1{\sqrt{8\pi}}\,\lambda\Big(B^{d_{G,t}}(0,1)\Big)\,e^{-\delta^2_{G,t}(x,0)}.
\end{equation}
Comparing \eqref{intro-e08} and \eqref{intro-e14} we note that $p_t$ and $\gauss_t$ have the same structure.

The central purpose of this paper is to work out that we should expect many symmetric L\'evy processes to have a density which is completely determined by two one-parameter families of metrics by a formula which is completely analogous to \eqref{intro-e08} or \eqref{intro-e14}. We will provide large classes of examples and discuss some consequences.

Throughout this paper $(X_t)_{t\geq 0}$ will be a symmetric L\'evy process with state space $\rn$. It is well known that its characteristic function is given by
\begin{equation*}
    \Ee e^{i\xi\cdot X_t} = e^{-t\psi(\xi)}
\end{equation*}
where $\psi:\rn\to\comp$ is the characteristic exponent. The exponent $\psi$ has a L\'evy-Khintchine representation,
\begin{equation*}
    \psi(\xi)
    = i\ell\cdot\xi + \frac 12 \sum_{j,k=1}^n q_{jk}\xi_j\xi_k + \int_{y\neq 0} \left(1-e^{iy\cdot\xi} + \frac{iy\cdot\xi}{1+|y|^2}\right)\nu(dy)
\end{equation*}
where $\ell\in\rn$, $(q_{jk})_{j,k}\in\real^{n\times n}$ is a symmetric positive semidefinite matrix and $\nu$ is a Borel measure on $\rn\setminus\{0\}$ such that $\int_{y\neq 0} \left(1\wedge |y|^2\right)\,\nu(dy)<\infty$. This means that $\psi$ is a continuous negative definite function in the sense of Schoenberg. If $(X_t)_{t\geq 0}$ is symmetric, then $\psi(\xi)\geq 0$, $\psi(\xi)=\psi(-\xi)$ and $\xi\mapsto\sqrt{\psi(\xi)}$ is subadditive, i.e.
$$
    \sqrt{\psi(\xi+\eta)}\leq\sqrt{\psi(\xi)} + \sqrt{\psi(\eta)}.
$$
Thus, if we require that $\psi(\xi)=0$ if, and only if, $\xi=0$, then $(\xi,\eta)\mapsto \sqrt{\psi(\xi-\eta)}$ generates a metric on $\rn$ and we can understand $\big(\rn,\sqrt{\psi},\lambda\big)$ as a metric measure space.

We will always assume that $e^{-t\psi}\in L^1(\rn)$ which implies that the probability distribution $p_t$ of each $X_t$, $t>0$, has a density with respect to Lebesgue measure $\lambda$; we will denote these densities again by $p_t(x)$.

The first important observation is Theorem \ref{pto-03} which tells us that
\begin{equation}\label{intro-e16}
    p_t(0)
    = (2\pi)^{-n}\int_\rn e^{-t\psi(\xi)}\,d\xi
    = (2\pi)^{-n}\int_0^\infty \lambda\Big(B^{d_\psi}\big(0,\sqrt{r/t}\big)\Big)\,e^{-r}\,dr
\end{equation}
where $B^{d_\psi}(0,r)$ denotes the ball with centre $0$ and radius $r>0$ with respect to the metric $d_\psi(x,y)=\sqrt{\psi(x-y)}$. A first version of this result was already proved in \cite{KS1}. While \eqref{intro-e16} is an exact formula, we get
\begin{equation}\label{intro-e17}
    p_t(0) \asymp \lambda\Big(B^{d_\psi}\big(0, 1/\sqrt t\big)\Big)
\end{equation}
whenever the metric measure space $\big(\rn,d_\psi,\lambda\big)$ has the volume doubling property. (By $f\asymp g$ we mean that there exists a constant $0<\kappa<\infty$ such that $\kappa^{-1} f(x) \leq g(x) \leq \kappa f(x)$ for all $x$.) Introducing the one-parameter family of metrics $d_{\psi,t}(\cdot,\cdot)$ by
\begin{equation*}
    d_{\psi,t}(\xi,\eta)
    := \sqrt{t\,\psi(\xi-\eta)}
\end{equation*}
we find, if \eqref{intro-e17} holds, that
\begin{equation}\label{intro-e19}
    p_t(0) \asymp \lambda\Big(B^{d_{\psi,t}(\cdot,\cdot)}\big(0,1\big)\Big).
\end{equation}
In order to prove \eqref{intro-e17} or \eqref{intro-e19} we first need to understand the metric measure space $\big(\rn,d_\psi,\lambda\big)$. This is done in Section \ref{cndf}. Following some basic definitions we provide conditions for the doubling property to hold and we discuss when $\big(\rn,d_\psi,\lambda\big)$ is a space of homogeneous type. Note that only in case of Brownian motion we can expect that $\big(\rn,d_\psi,\lambda\big)$ is a length space. A more detailed study is devoted to the case of subordination, i.e.\ when $\psi$ is the composition of a Bernstein function $f$ (the Laplace exponent of an increasing L\'evy process) and a continuous negative definite function (characteristic function of a L\'evy process) $\phi$. The most interesting case is $\psi(\xi) = f(|\xi|^2)$. We discuss several examples and these are used to illustrate \eqref{intro-e19}.

In order to understand the behaviour of $p_t(x)$ for $x\neq 0$ we observe that
\begin{equation*}
    \frac{p_t(x)}{p_t(0)}
    = \int_\rn e^{-ix\cdot\xi}\,\frac{e^{-t\psi(\xi)}}{p_t(0)}\,d\xi
\end{equation*}
is the Fourier transform of a probability measure. The question is whether we can write $p_t(x)/p_t(0)$ as
\begin{equation}\label{intro-e21}
    \frac{p_t(x)}{p_t(0)} = e^{-\delta_{\psi,t}^2(x,0)}
\end{equation}
with a suitable one-parameter family of metrics $\delta_{\psi,t}(x,y)$, $t>0$. Section \ref{ptx} explains this idea in more detail and first examples are given. Our approach is not just an `educated guess'. A theorem of Schoenberg---in a formulation suitable for our discussion---states that a metric space $(\rn,d)$ can be isometrically embedded into an (in general infinite-dimensional) Hilbert space $\mathcal H$ if, and only if, $d(x,y)=\sqrt{\psi(x-y)}$ for some suitable continuous negative definite function $\psi:\rn\to\real$. Using the Dirichlet form and the carr\'e du champ associated with the L\'evy process $(X_t)_{t\geq 0}$ we outline the proof of  the fact that the metric space $(\rn,d_\psi)$ can isometrically be embedded into a Hilbert space; this is the part of Schoenberg's result which is important for our considerations.
Our general guide for the investigations in this paper is the rough idea that Fourier transforms of Gaussians are Gaussians---also in Hilbert spaces. Thus, we might consider to obtain $p_t(x)$ or $p_t(x)/p_t(0)$ as pre-image of Fourier transforms of Gaussians in $\mathcal H$. So far, we did not succeed to formalize this idea, however, already during the \emph{3rd Conference on L\'evy processes: Theory and Applications} 2003 in Paris the first-named author launched this idea to use this correspondence to study L\'evy processes picking up some work of P.A.\ Meyer \cite{meyer}.

So far we have only partial answers for \eqref{intro-e21} to hold. In Section~\ref{dist} we begin with the density of a single random variable, i.e.\ we will not take into account that they belong to the transition function of a process. However, we assume that they are infinitely divisible random variables, hence they can always be embedded into the transition function of a L\'evy process. We introduce in Section \ref{dist} the class $\Nsf$ of infinitely divisible probability distributions consisting of those $p$ for which $\mathcal F^{-1}\big[\frac p{p(0)}\big]$ is again infinitely divisible. Thus, if for a L\'evy process $(X_t)_{t\geq 0}$ the density $p_{t_0}$ belongs to $\Nsf$, then $p_{t_0}$ satisfies \eqref{intro-e21}. We give large classes of examples including Fourier self-reciprocal densities, generalized hyperbolic distributions and more.

In Section \ref{sub} and \ref{proc} we return to our investigations on processes.  As so often when dealing with L\'evy processes, subordinate Brownian motion plays a distinguished role. In Section~\ref{sub} we present some results of general nature. We prove in Theorem~\ref{sub-15} (for $n=1$) that if  $\Ff^{-1}[p_t] (\xi)=e^{-tf(|\xi|)}$ for some Bernstein function $f$ such that $f(0)=0$ and $\int_0^\infty e^{-tf(r)}dr<\infty$, then $p_t(x)=p_t(0)e^{-g_t(|x|^2)}$ for a suitable family of Bernstein functions $g_t$; of course, $\sqrt{g_t(x-y)}$ gives a metric on $\real$. Although the theorem is proved only for $n=1$, its proof extends to $n=2$ and $n=3$.

In Section~\ref{proc} we discuss further examples of processes for which \eqref{intro-e21} holds. These examples are processes with transition functions which are certain mixtures of Gaussians. While our examples already indicate the scope of our approach, their proofs depend essentially on the special structures of the underlying transition densities. So far we do not have a proof for our general
\begin{conjecture}
Let $\mcn(\rn)$ denote the continuous negative definite functions that induce a metric on $\rn$ which generates the Euclidean topology. If $\psi\in  \mcn(\rn)$ and $e^{-t\psi}\in L^1(\rn)$, then there exists a one-parameter family of metrics $\delta_t(\cdot,\cdot)$ such that
\begin{equation*}
    p_t(x)=p_t(0)e^{-\delta^2_t(x,0)}
\end{equation*}
holds.
\end{conjecture}
We emphasize that we are looking for a metric $\delta_t(\cdot,\cdot)$, and we do not require that it is of the type $\sqrt{\psi_t(x-y)}$ where $\psi_t(\xi)$ is a family of continuous negative definite functions indexed by $t>0$. An interesting remark was made by Rama Cont, namely to investigate whether the metric $\delta_t(\cdot,\cdot)$ can be related to a good rate function for large deviations as it is the case for diffusions.

In the final Section~\ref{fel} we give a brief outline of the situation when the L\'evy process is replaced by a Feller process generated by a pseudo-differential operator with a negative definite symbol, compare \cite{H95}--\cite{H98b} and \cite{J94} as well as  \cite{J2} for large classes of examples. We will have to work with metrics varying with the current position in space as it is the case in (sub-)Riemannian geometry. However, the fact that we cannot expect to work in length spaces causes serious problems when we try to understand the underlying geometry.

We would like to mention more recent work in which estimates for heat kernels are obtained when starting with a metric measure space having the volume doubling property: M.\ Barlow, A.\ Grigor'yan and T.\ Kumagai \cite{bar-gri-kum}, Z.-Q.\ Chen and T.\ Kumagai \cite{che-kum,che-kum08}, A.\ Grigor'yan and J.\ Hu \cite{gri-hu}, and A.\ Grigor'yan, J.\ Hu and K.-S.\ Lau \cite{gri-hu-lau,gri-hu-lau10}, to mention some of this work. Note the difference to our point of view. The metric measure space considered by us is induced by the characteristic exponent \emph{on the Fourier space}, not on the state space. Our conjecture, here proved for many classes of processes, is that we can also find a \emph{new metric on the state space} which will yield a Gaussian estimate when combined with the metric induced by the characteristic exponent which gives the diagonal term. For L\'evy and L\'evy-type processes this seems to be the natural approach.

\begin{notation}
In general, we follow our monographs \cite{J}--\cite{J3} and \cite{SSV}. In particular, we use $\Ff u(\xi) = \widehat u(\xi) = (2\pi)^{-n}\int_\rn e^{i\xi x} u(x)\,dx$ for the Fourier transform and we write $\mathscr S(\rn)$ for the Schwartz space. By $K_\lambda$ we denote the Bessel functions of the third kind, cf.\ \cite[vol.\ 2]{erd-et-al}. We write $X\sim Y$ if two random variables $X$ and $Y$ have the same probability distribution and $X\sim \mu$ means that $X$ has the probability distribution $\mu$. If $f$ and $g$ are functions, $f\asymp g$  means that  there exists a constant $\kappa$ such that $\kappa^{-1} f(x)\leq g(x)\leq\kappa \,f(x)$ holds for all $x$, and $f\approx g$, $x\to a$, stands for $\lim_{x\to a} f(x)/g(x) = 1$. All other notations are standard or explained in the text.
\end{notation}

\medskip
We dedicate this paper to Professor Mu-Fa Chen and Professor Zhi-Ming Ma in appreciation of their outstanding contributions to mathematics and their remarkable success of building up in China one of the world's finest centres in probability theory.

\begin{ack}
The authors would like to thank Bj\"orn B\"ottcher and Walter Hoh for comments made while working on this paper.
\end{ack}

\section{Auxiliary results}\label{aux}

\subsection*{Fourier transforms and characteristic functions}
The \emph{Fourier transform} of a bounded Borel measure $\mu$ on $\rn$ is defined by
\begin{equation}\label{aux-e02}
    \Ff\mu(\xi) = (2\pi)^{-n}\int_\rn e^{-ix\xi}\,\mu(dx),\quad\xi\in\rn.
\end{equation}
By \emph{Bochner's theorem} the Fourier transform is a bijective and  bi-continuous mapping from the cone of bounded Borel measures (equipped with the weak topology) to the cone of continuous positive definite functions (equipped with the topology of locally uniform convergence). By linearity we can extend \eqref{aux-e02} to signed measures; for $u\in\mathscr S(\rn)$ we get the classical formulae for the (inverse) Fourier transform
\begin{equation*}
    \Ff u(\xi)
    = (2\pi)^{-n}\int_\rn e^{-ix\xi} u(x)\,dx
    \quad\text{and}\quad
    \Ff^{-1}v(\eta) = \int_\rn e^{iy\eta} v(y)\,dy.
\end{equation*}
Obviously, $\Ff^{-1}$ extends canonically to the bounded Borel measures
\begin{equation*}
    \Ff^{-1}\mu(\eta) = \int_\rn e^{iy\eta}\,\mu(dy),\quad \eta\in\rn.
\end{equation*}
If $\mu$ is the probability law of a random variable $Y$, $\Ff^{-1} \mu(\eta)$ is the \emph{characteristic function} $\chi_Y(\eta) = \Ee e^{i\eta Y}$.

With our normalization of the Fourier transform \emph{Plancherel's theorem} becomes
\begin{equation*}
    \nnorm u_{L^2}
    = (2\pi)^n \nnorm{\Ff u}_{L^2}
    \quad\text{and}\quad
    \int_{\rn} u(x)\,\mu(dx)
    = (2\pi)^n \int_{\rn} \Ff u(\xi) \, \overline{\Ff\mu(\xi)}\,d\xi.
\end{equation*}
Whenever convolution and Fourier transforms of $u$ and $v$ are defined, the \emph{convolution theorem} holds, i.e.\
\begin{equation*}
    \Ff^{-1}(u\star v) = \Ff^{-1} u \cdot \Ff^{-1} v
    \quad\text{and}\quad
    \Ff(u\cdot v) = \Ff u\star\Ff v.
\end{equation*}

\subsection*{Infinite divisibility}
A probability measure $\mu$ is called an \emph{infinitely divisible probability distribution} if for every $n>0$ there exists some probability measure $\mu_n$ such that $\mu = \mu_n^{\star n} = \mu_n \star\ldots\star\mu_n$ ($n$ factors). Let $X$ be a random variable with law $X$. Then the following statements are equivalent to saying that $\mu$ is infinitely divisible:
\begin{enumerate}[\quad\upshape 1)]
\item
    the random variable satisfying $X\sim\mu$ is an \emph{infinitely divisible random variable}, i.e.\ for every $n\in\nat$ there exist independent and identically distributed random variables $X_1, \ldots, X_n$ such that $X\sim X_1+\cdots+X_n$.
\item
    the characteristic function $\chi_X(\xi) = \Ee e^{i\xi X} = \Ff^{-1}\mu(\xi)$ of the random variable $X$ is an \emph{infinitely divisible characteristic function}, i.e.\ for every $t>0$ the function $(\chi_X)^t$ is again the characteristic function of some random variable.
\end{enumerate}
It is a classical result that $X$ or $\mu$ are infinitely divisible if, and only if, the log-characteristic function $\psi(\xi) := -\ln\chi_X(\xi) = -\ln\Ff^{-1}\mu(\xi)$ is a \emph{continuous negative definite function} (in the sense of Schoenberg). These functions have a unique L\'evy-Khintchine representation, i.e.\
\begin{equation}\label{aux-e13}
    \psi(\xi)
    = i\ell\cdot\xi + \frac 12 \sum_{j,k=1}^n q_{jk}\xi_j\xi_k + \int_{y\neq 0} \left(1-e^{iy\cdot\xi} + \frac{iy\cdot\xi}{1+|y|^2}\right)\nu(dy)
\end{equation}
where $\ell\in\rn$, $(q_{jk})_{j,k}\in\real^{n\times n}$ is a symmetric positive semidefinite matrix and $\nu$ is a Borel measure on $\rn\setminus\{0\}$ such that $\int_{y\neq 0} \left(1\wedge |y|^2\right)\,\nu(dy)<\infty$.

\subsection*{Convolution semigroups}
A \emph{convolution semigroup} $(\mu_t)_{t\geq 0}$ on $\rn$ is a family of probability measures $\mu_t$ defined on $\rn$ satisfying
$$
    \mu_s \star \mu_t = \mu_{s+t},\; s,t\geq 0,
    \quad\text{and}\quad
    \mu_0 = \delta_0.
$$
We will always assume that $(\mu_t)_{t\geq 0}$ is \emph{vaguely continuous}, i.e.\
$$
    \lim_{t\to 0} \int_\rn u(x)\,\mu_t(dx) = u(0)
    \quad\text{for all}\quad u\in C_c(\rn).
$$
It follows from the definition that each measure $\mu_t$, $t>0$, is infinitely divisible. Therefore every (vaguely continuous) convolution semigroup $(\mu_t)_{t\geq 0}$ on $\rn$ is uniquely characterized by a continuous negative definite function $\psi$ such that
\begin{equation}\label{aux-e14}
    \Ff^{-1}\mu_t(\xi) = e^{-t\psi(\xi)},\quad t\geq 0,\;\xi\in\rn.
\end{equation}
Conversely, every continuous negative definite function $\psi$ with $\psi(0)=0$ determines, by \eqref{aux-e14}, a unique (vaguely continuous) convolution semigroup $(\mu_t)_{t\geq 0}$ on $\rn$.

If $\mu$ is an infinitely divisible probability distribution, there is a unique vaguely continuous convolution semigroup $(\mu_t)_{t\geq 0}$ such that $\mu_1 = \mu$; indeed, $\Ff^{-1}\mu_t = (\Ff^{-1}\mu)^t$. Conversely, if $\mu_1$ is an element of $(\mu_t)_{t\geq 0}$, then $\mu_1$ is infinitely divisible.

\subsection*{Subordination}
Subordination in the sense of Bochner is a method to obtain new convolution semigroups from a given one. Let $(\eta_t)_{t\geq 0}$ be a convolution semigroup on $\real$ where all measures $\eta_t$, $t\geq 0$, are supported in $[0,\infty)$. Since $\supp\eta_t\subset[0,\infty)$ it is more convenient to describe $\eta_t$ in terms of the (one-sided) Laplace transform. Similar to \eqref{aux-e14} we see that
\begin{equation}\label{aux-e18}
    \mathcal L\eta_r(\lambda) = \int_{[0,\infty)} e^{-\lambda t}\,\eta_r(t) = e^{-rf(\lambda)}.
\end{equation}
The characteristic (Laplace) exponent $f$ is a \emph{Bernstein function}, i.e.\ $f\in C^\infty(0,\infty)$ such that $f\geq 0$ and $(-1)^{k-1} f^{(k)}\geq 0$ for all $k\geq 1$. All Bernstein functions have a unique representation
\begin{equation}\label{aux-e20}
    f(\lambda) = a + b\lambda + \int_{(0,\infty)} (1-e^{-\lambda t})\,\gamma(dt)
\end{equation}
where $a,b\geq 0$ and $\gamma$ is  a Borel measure on $(0,\infty)$ satisfying $\int_{(0,\infty)} \left(1\wedge t\right)\,\gamma(dt) < \infty$. The triplet $(a,b,\gamma)$, the Bernstein function $f$ and the one-sided convolution semigroup $(\eta_t)_{t\geq 0}$ are, because of \eqref{aux-e18} and \eqref{aux-e20}, in one-to-one correspondence.

Let $(\mu_t)_{t\geq 0}$ and $(\eta_t)_{t\geq 0}$, $\supp\eta_t\subset [0,\infty)$, be convolution semigroups on $\rn$ and $\real$, respectively. Then the following integrals (convergence in the vague topology)
\begin{equation}\label{aux-e22}
    \mu_t^f := \int_{[0,\infty)} \mu_s\,\eta_t(ds),\quad t\geq 0,
\end{equation}
define a new convolution semigroup on $\rn$, $(\mu_t^f)_{t\geq 0}$, which is called the \emph{subordinate semigroup}. The characteristic function of the sub-probability measure $\mu_t^f$ is given by
\begin{equation*}
    \Ff^{-1}\mu_t^f(\xi) = e^{-tf(\psi(\xi))}.
\end{equation*}
In fact, $f\circ\psi$ is, for every Bernstein function $f$, again a continuous negative definite function. Note that the Bernstein functions are the only functions that operate on the continuous negative definite functions in the sense that $f\circ\psi$ is continuous negative definite whenever $\psi$ is, cf.\ \cite{J1}.

\subsection*{Mixtures}
The probability measure $\mu_t^f$ defined in the formula \eqref{aux-e22} may be understood as a \emph{mixture} of the probability measures $(\mu_s)_{s\geq 0}$ under the mixing probability measure $\eta_t(ds)$. More generally, let $(\pi(s;\cdot))_{s\in\real}$ be a family of probability measures on $\rn$ and assume that $\rho$ is a probability measure on the parameter space $\real$. Then
\begin{equation}\label{aux-e26}
    \pi^\rho(B) := \int_\real \pi(s;B) \rho(ds),\quad B\subset\rn\text{\ \ Borel},
\end{equation}
is again a probability measure on $\rn$.

\bigskip
Our standard references for the Fourier transform are the monographs
 \cite{J1} and Berg--Forst \cite{BF}; for Bernstein functions and related topics we refer to \cite{SSV}.
  Basic notions from probability theory can be found in Breiman \cite{B}, mixtures of
  probability measures are discussed in Sato \cite{Sato} and in Steutel--van Harn \cite{SH04}. 

\section{Metric measure spaces and negative definite functions}\label{cndf}

Recall that a \emph{metric measure space} is a triple $(X,d,\mu)$ where $(X,d)$ is a metric space and $\mu$ is a measure on the Borel sets of the space $X$. A good introduction to the analysis on metric measure spaces is the book by Heinonen \cite{H}.

We are mainly interested in metric measure spaces whose metric is induced by a negative definite function. Our basic reference for negative definite functions and their properties is \cite{J1}. Let $\psi:\rn\to\comp$ be a locally bounded negative definite function. Then
\begin{equation*}
    |\psi(\xi)| \leq c_\psi (1+|\xi|^2)
    \quad\text{with}\quad
    c_\psi = 2\sup_{|\eta|\leq 1}|\psi(\eta)|
\end{equation*}
and
\begin{equation}\label{cndf-e04}
    \sqrt{|\psi(\xi+\eta)|}
    \leq \sqrt{|\psi(\xi)|} + \sqrt{|\psi(\eta)|}.
\end{equation}
In particular, whenever $\psi(\xi_0)=0$ for some $\xi_0\neq 0$, $\psi$ is periodic with period $\xi_0$.

Since $\psi(-\xi) = \overline{\psi(\xi)}$, the map $\xi\mapsto |\psi(\xi)|$ is even; in view of \eqref{cndf-e04} it is easy to see that every locally bounded, non-periodic negative definite function with $\psi(0)=0$ induces a metric on $\rn$ by
\begin{equation*}
    d_\psi:\rn\times\rn \to [0,\infty),\quad d_\psi(\xi,\eta) := \sqrt{|\psi(\xi-\eta)|}.
\end{equation*}
The metric $d_\psi$ is invariant under translations, i.e.
$$
    d_\psi(\xi+\zeta,\eta+\zeta) = d_\psi(\xi,\eta).
$$
Therefore Lebesgue measure $\lambda$ is the canonical choice if we consider $(\rn,d_\psi,\lambda)$ as a metric measure space.

We denote by
\begin{equation*}
\begin{aligned}
    B^{d_\psi}(\xi,r)
    &:= \big\{\eta\in\rn \::\: d_\psi(\xi,\eta) < r \big\}
     = \big\{\eta\in\rn \::\: |\psi(\xi-\eta)| < r^2 \big\}\\
    K^{d_\psi}(\xi,r)
    &:= \big\{\eta\in\rn \::\: d_\psi(\xi,\eta) \leq r \big\}
     = \big\{\eta\in\rn \::\: |\psi(\xi-\eta)| \leq r^2 \big\}
\end{aligned}
\end{equation*}
the open and closed balls with radius $r>0$ and centre $\xi$ in the metric space $(\rn,d_\psi)$. Note that $B^{d_\psi}(\xi,r) = \xi + B^{d_\psi}(0,r)$. In general, $B^{d_\psi} \subsetneqq \overline{B^{d_\psi}} \subsetneqq K^{d_\psi}$. A typical counterexample can be constructed using continuous negative definite functions of P\'olya-type. For example, if we set  $\phi(\xi):=|\xi|\wedge 1$, $\xi\in\real$, then $B^{d_\phi}(0,1) = (-1,1)$, $\overline{B^{d_\phi}}(0,1) = [-1,1]$ and $K^{d_\phi}(0,1)=\real$.

In order to compare the metric $d_\psi$ with the usual Euclidean metric we define
\begin{equation}\label{cndf-e10}\begin{aligned}
        m(r) &:= \inf\left\{ |\eta|  \::\: \sqrt{|\psi(\eta)|} = r\right\},
        \\
        M(r) &:= \sup\left\{ |\eta|  \::\: \sqrt{|\psi(\eta)|} = r\right\}.
\end{aligned}\end{equation}
Clearly, $0\leq m(r)\leq M(r)\leq\infty$ are the maximal resp.\ minimal radii of Euclidean balls such that
$$
    B(\xi,m(r)) \subset B^{d_\psi}(\xi,r) \subset B(\xi,M(r))
$$
holds.
\begin{lemma}\label{cndf-03}
    Let $\psi:\rn\to\comp$ be a non-periodic continuous negative definite function with $\psi(0)=0$. Then
    $$
        0 < m(r) \leq M(r) < \infty
        \quad\text{for all}\quad
        0 < r < \liminf_{|\xi|\to\infty} \sqrt{|\psi(\xi)|}.
    $$
    Moreover $m$ and $M$ are monotonically increasing, and the following assertions are equivalent:
    \begin{enumerate}[\qquad\upshape a)]
        \item $M(2r)/m(r) \leq c_2$ for all $r>0$;
        \item $M(\gamma r)/m(r) \leq c_\gamma$ for all $r> 0$ and all $\gamma > 1$;
        \item $M(\delta r)/m(r) \leq c_\delta$ for all $r>0$ and some $\delta > 1$.
    \end{enumerate}
\end{lemma}
\begin{proof}
    Since $\psi$ is non-periodic and continuous, $0 < r < \liminf_{|\xi|\to\infty} \sqrt{|\psi(\xi)|}$ implies that the level sets $\left\{\xi\::\: |\psi(\xi)| = r^2\right\}$ are non-empty and compact (in the Euclidean topology). Therefore $0<m(r)\leq M(r)<\infty$.

    Now let $0<r<R<\liminf_{|\xi|\to\infty} |\psi(\xi)|$ and pick some $\xi_R\in\rn$ such that $|\psi(\xi_R)| = R^2$ and $m(R)=|\xi_R|$. Consider the curve $\gamma(t):=|\psi(t\xi_R)|$, $0\leq t\leq 1$. By assumption, $t\mapsto\gamma(t)$ is continuous, $\gamma(0)=0$ and $\gamma(1) = R^2$. Therefore, there exists some $\theta=\theta_r\in (0,1)$ such that $\gamma(\theta) = r^2$. Since $|\theta\xi_R|<|\xi_R|$ and $|\psi(\theta\xi_R)|=r^2$, we conclude that
    \begin{gather*}
        m(r) \leq |\theta\xi_R|<m(R).
    \end{gather*}
    The proof that $M$ is monotone is similar.

    Let us now turn to the assertions a)--c). Clearly, b) implies a). Since $r\mapsto M(r)$ is increasing, it is enough to show that a) entails b) for $\gamma > 2$. If $\gamma > 2$ we can uniquely write it in the form $\gamma = 2^k \gamma_0$ where $k\in\nat$ and $1 \leq \gamma_0< 2$. Thus,
    $$
        \frac{M(\gamma r)}{m(r)}
        = \prod_{j=1}^k \frac{M(2^j\gamma_0 r)}{M(2^{j-1}\gamma_0 r)} \frac{M(\gamma_0 r)}{m(r)}
        \leq \prod_{j=1}^k \frac{M(2^j\gamma_0 r)}{m(2^{j-1}\gamma_0 r)} \frac{M(2 r)}{m(r)}
        \leq c_2^{k+1}.
    $$
    The direction b)$\Rightarrow$c) is obvious; the converse follows as in the case where $\delta=2$.
\end{proof}

\begin{lemma}\label{cndf-05}
    Let $\psi:\rn\to\comp$ be a continuous negative definite function. Then the closed ball $K^{d_\psi}(0,r)$, $r>0$, is bounded in the Euclidean topology if, and only if, $r^2 < \liminf_{|\xi|\to\infty} |\psi(\xi)|$. Moreover, $d_\psi$ generates on $\rn$  the Euclidean topology if, and only if, $\liminf_{|\xi|\to\infty} |\psi(\xi)|>0$.
\end{lemma}
\begin{proof}
    Assume that $r^2 < \liminf_{|\xi|\to\infty} |\psi(\xi)|$. Then $M(r)= \sup\left\{|\xi|\::\: \sqrt{|\psi(\xi)|} = r\right\}$ is finite, and the inclusion $K^{d_\psi}(0,r)\subset \overline{B(0,M(r))}$ shows that $K^{d_\psi}(0,r)$ is bounded in the Euclidean topology.

    Conversely, assume that $K^{d_\psi}(0,r)$ is bounded. Then $K^{d_\psi}(0,r)\subset B(0,\rho)$ for some $\rho>0$. In particular, $|\psi(\xi)| > r^2$ for all $|\xi|>\rho$. This shows that $\liminf_{|\xi|\to\infty} |\psi(\xi)| > r^2$.

    If $\liminf_{|\xi|\to\infty} |\psi(\xi)|>0$, then we have for all $0 < r < \liminf_{|\xi|\to\infty} \sqrt{|\psi(\xi)|}$
    $$
        B(0,m(r)) \subset B^{d_\psi}(0,r) \subset \overline{B(0,M(r))}.
    $$
    This proves that the neighbourhood basis induced by $d_\psi$ and the Euclidean neighbourhood basis are comparable, i.e.\ the topologies coincide.

    If $\liminf_{|\xi|\to\infty}|\psi(\xi)|=0$ and if $\psi$ is not periodic, then the $d_\psi$ metric cannot distinguish between $0$ and the points at infinity. This means that the metric $d_\psi$ does not generate the Euclidean topology.
\end{proof}

For our purposes it is helpful to assume that $d_\psi$ generates on $\rn$ the Euclidean topology. To simplify notation we introduce the following definition.

\begin{definition}\label{cndf-07}
    Let $\psi:\rn\to\comp$ be a non-periodic, locally bounded negative definite function with $\psi(0)=0$. We call $\psi$ \emph{metric generating} on $\rn$, if the metric $d_\psi(\xi,\eta):=\sqrt{|\psi(\eta-\xi)|}$ generates on $\rn$ the Euclidean topology.  The set of all continuous metric generating negative definite functions on $\rn$ is denoted by $\mcn(\rn)$.
\end{definition}
We will use the term `metric generating' exclusively for $\psi\in\mcn(\rn)$.

From Lemma \ref{cndf-05} it follows that a continuous negative definite function is metric generating if, and only if, $\liminf_{|\xi|\to\infty}|\psi(\xi)|>0$.

Let $\psi\in\mcn(\rn)$. We want to study the metric measure space $(\rn,d_\psi,\lambda)$. In the analysis on metric measure spaces the notion of volume doubling plays a central role.
\begin{definition}\label{cndf-09}
    Let $(X,d,\mu)$ be a metric measure space. We say that $(X,d,\mu)$ or $\mu$ has the \emph{volume doubling property} if there exists a constant $c_2$ such that
    \begin{equation}\label{cndf-e20}
        \mu(B^d(x,2r)) \leq c_2 \, \mu(B^d(x,r))
    \end{equation}
    holds for all metric balls $B^d(x,r) = \{y\in X\::\: d(y,x)<r\}\subset X$. If \eqref{cndf-e20} holds only for all balls with radii $r<\rho$ for some fixed $\rho>0$, we say that $(X,d,\mu)$ (or $\mu$) is \emph{locally} volume doubling.
\end{definition}

\begin{remark}\label{cndf-11}
    If $(X,d,\mu)$ is volume doubling with doubling constant $c_2>1$, then it follows for every $R\geq 1$ that
    \begin{equation}\label{cndf-e22}
        \mu\big(B^d(x,R)\big)
        \leq c_2^{\log_2 R} \mu\big(B^d(x,1)\big)
        = R^{\log_2 c_2} \mu\big(B^d(x,1)\big).
    \end{equation}
    Thus, volume doubling entails that balls have at most power growth of their volume.
\end{remark}

Recall that a metric space $(X,d)$ is said to be of \emph{homogeneous type} in the sense of Coifman and Weiss \cite{CW} if there exists some $N\geq 1$ such that for all $x\in X$ and all radii $r>0$ the ball $B(r,x)$ contains at most $N$ points $x_1, \ldots, x_N$ such that $d(x_j,x_k)>\frac r2$ whenever $j\neq k$.

The following result is taken from \cite{SL}.
\begin{lemma}\label{cndf-13}
    If $(\rn,d_\psi,\lambda)$, $\psi\in\mcn(\rn)$, is volume doubling, it is of homogeneous type.
\end{lemma}
\begin{proof}
    Let $c_2$ be the volume doubling constant as in \eqref{cndf-e20} and let $x_1,\ldots,x_N\in B^{d_\psi}(x,r)$ such that $d_\psi(x_j,x_k)>\frac r2$ for $j\neq k$. By the triangle inequality we see
    $$
        B^{d_\psi}(x_j,r/4)\cap B^{d_\psi}(x_k,r/4)= \emptyset
        \quad\text{for all}\quad j\neq k.
    $$
    Since $d_\psi$ is invariant under translations we get $\lambda\big(B^{d_\psi}(x_j,r/4)\big) = \lambda\big(B^{d_\psi}(x,r/4)\big)$. Moreover, $B^{d_\psi}(x_j,r/4) \subset B^{d_\psi}(x,2r)$ and applying \eqref{cndf-e20} three times yields
    $$
        N \lambda\left(B^{d_\psi}(x,r/4)\right)
        =\lambda\left(\bigcup_{j=1}^N B^{d_\psi}(x_j,r/4)\right)
        \leq\lambda\left(B^{d_\psi}(x,2r)\right)
        \leq c_2^3 \lambda\left(B^{d_\psi}(x,r/4)\right).
    $$
    This proves that $N\leq c_2^3$.
\end{proof}

We are interested in the volume growth of balls in the metric measure space $(\rn,d_\psi,\lambda)$. The following result appears in in a weaker form in \cite{SL}.
\begin{proposition}\label{cndf-15}
    Let $\psi\in\mcn(\rn)$ and $m,M$ as in \eqref{cndf-e10}. Then the following inequality holds for all $0<r<R<\liminf_{|\xi|\to\infty}\sqrt{|\psi(\xi)|}$
    \begin{equation}\label{cndf-e24}
        \lambda\big(B^{d_\psi}(0,R)\big)
        \leq \left(\tfrac{M(R)}{m(r)}\right)^n \lambda\big(B^{d_\psi}(0,r)\big).
    \end{equation}
\end{proposition}

Note that \eqref{cndf-e24} does \emph{not} imply the doubling property as $M(2r)/m(r)$ may depend on $r$.

\begin{proof}
    For $0<r<R<\liminf_{|\xi|\to\infty}\sqrt{|\psi(\xi)|}\leq\infty$ the functions $m$ and $M$ are strictly positive and finite.  By the very definition of the functions $m$ and $M$ we see that
    $$
        B^{d_\psi}(0,R) \subset B(0,M(R)) = \frac{M(R)}{m(r)} B(0,m(r)) \subset \frac{M(R)}{m(r)} B^{d_\psi}(0,r).
    $$
    Taking Lebesgue measure in the above chain of inclusions yields \eqref{cndf-e24}.
\end{proof}

Many of the most important and concrete continuous negative definite functions are related to subordination. Let $f$ be a (non-degenerate, i.e.\ non-constant) Bernstein function and $\psi\in\mcn(\rn)$. If $f(0)=0$, then $f\circ\psi$ is again metric generating since every Bernstein function is strictly monotone increasing, cf.\ \cite[Remark 1.5 and Definition 3.1]{SSV}. For some subordinate negative definite functions we can calculate the volume growth constant appearing in \eqref{cndf-e22} explicitly. Assume that $\psi(\xi) = f(|\xi|^2)$. Since $f$ is strictly increasing, it is obvious that
$$
    M(r) = m(r) = \sqrt{f^{-1}(r^2)}.
$$
Therefore the volume growth constant is $(R/r)^{n/\alpha}$ if $\psi(\xi) = |\xi|^{2\alpha}$ with $0<\alpha<1$. In this case we do have volume doubling. If, however, $\psi(\xi) = \ln (1+|\xi|^2)$, the constant becomes $\left(\frac{\exp(R^2)-1}{\exp(r^2)-1}\right)^{n/2}$ and volume doubling clearly fails. Note that we also do not have power growth. Finally, if $\psi(\xi) = 1-\exp(-|\xi|^2)$, the inverse $f^{-1}$ is only defined on $(0,1)$. Therefore, the volume growth constant is only defined for radii $r<R<1$ and we get $\left(\frac{\ln(1-R^2)}{\ln(1-r^2)}\right)^{n/2}$. In this case we have local volume doubling.

Let us close this section with an observation from \cite[Lemma 9]{KS1} where it is shown that the volume doubling property for large radii entails that $(1+\psi)^{-\kappa/2}\in L^2(\real^n,\lambda)$ for some suitable exponent $\kappa>0$. This, in turn, has implications for the smoothness of the transition densities, cf.\ \cite{KS1}.

\begin{lemma}\label{cndf-19}
    Let $\phi : (0,\infty)\to (0,\infty)$ be an increasing function such that
    $$
        \liminf_{r\to\infty} \frac{\phi(Cr)}{\phi(r)} > 1
        \quad\text{for some}\quad C > 1.
    $$
    Then $\phi$ grows at least like a (fractional) power, i.e.\ there exist constants $c_0,r_0,\kappa>0$ such that
    $$
        \phi(r) \geq c_0\,r^\kappa\quad\text{for all}\quad r\geq r_0.
    $$
\end{lemma}
\begin{proof}
    By assumption there exist some $\gamma>1$ and $r_0 > 0$ such that
    $$
        \phi(Cr)\geq \gamma \phi(r)\quad\text{for all}\quad r\geq r_0.
    $$
    Let $r\in [C^k r_0, C^{k+1}r_0)$. Then we find
    $$
        \phi(r) \geq \phi(C^k r_0) \geq \gamma^k \phi(r_0).
    $$
    Let $\kappa$ be the unique solution of the equation $C = \gamma^{1/\kappa}$. Then we get $\gamma^k \leq r^\kappa/r_0^\kappa \leq \gamma\cdot \gamma^k$ and this entails that
    \begin{gather*}
        \phi(r) \geq \gamma^k \phi(r_0) \geq \frac 1\gamma\,\frac{\phi(r_0)}{r_0^\kappa}\,r^\kappa
        \quad\text{for all}\quad r\geq r_0.
    \qedhere
    \end{gather*}
\end{proof}

\begin{proposition}\label{cndf-21}
    Let $\psi(\xi) = f(|\xi|^2)$ be a continuous negative definite function where $f$ is a Bernstein function with $f(0)=0$. Then $\psi$ has the volume doubling property if, and only if, $\liminf_{r\to\infty} f(Cr)/f(r) > 1$ and $\liminf_{r\to 0} f(Cr)/f(r) > 1$ for some $C>1$.

    In particular, if $\psi$ has the volume doubling property, then $f(r)$ grows, as $r\to\infty$, at least like a fractional power.
\end{proposition}
\begin{proof}
    Since $\psi(\xi) = f(|\xi|^2)$, we see that $M(r) = m(r) = \sqrt{f^{-1}(r^2)}$ and $B^{d_\psi}(0,r)=B(0,m(r))$. Therefore, by Lemma \ref{cndf-03}, the volume doubling property of $\psi$ is the same as
    $$
        M(\gamma r) \leq c_\gamma m(r) \quad\text{for all}\quad r>0, \; \gamma>1
    $$
    which is equivalent to
    $$
        f^{-1}(\gamma^2 r) \leq c_\gamma^2 f^{-1}(r) \quad\text{for all}\quad r>0, \; \gamma>1.
    $$
    This means that the Bernstein function $f$ has to be unbounded and, consequently, bijective. Substituting in this inequality $r=f(x)$ and applying on both sides $f$ we get
    $$
        \liminf_{x\to 0} \frac{f(Cx)}{f(x)} > 1
        \quad\text{and}\quad
        \liminf_{x\to\infty} \frac{f(Cx)}{f(x)} > 1
    $$
    for some constant $C > 1$. Since $f$ is increasing, the second condition entails power growth, cf.\ Lemma \ref{cndf-19}.

    Conversely, if $\liminf_{r\to\infty} f(Cr)/f(r) > 1$ for some $C>1$, $f$ is unbounded (otherwise the limit inferior would be $1$) and $f$ is bijective. Therefore we can reverse the above argument to deduce volume doubling from $\liminf_{r\to\infty} f(Cr)/f(r) > 1$ and $\liminf_{r\to 0} f(Cr)/f(r) > 1$.
\end{proof}

\begin{corollary}\label{cndf-23}
    Let $\psi\in\mcn(\rn)$ and consider the volume function $v_\psi(r):=\lambda(B^{d_\psi}(0,r))$. If $\psi$ has the volume doubling property, $v^{-1}_\psi$ grows at least like a (fractional) power.
\end{corollary}
\begin{proof}
    Since $v_\psi$ is increasing, this is similar to the corresponding part of the proof of Proposition \ref{cndf-21}.
\end{proof}

Note that
$$
     v^{-1}_\psi(r) = \sup\big\{ t\geq 0\::\: \lambda \{\xi\in\rn\::\:  |\psi(\xi)|<t\}\leq r \big\}, \quad r>0.
$$
This means that $v^{-1}_\psi(r)$ is actually the \emph{increasing rearrangement} of $|\psi|$ which we denote by $\psi_*(r):= v^{-1}_\psi(r)$. With this notation Corollary \ref{cndf-23} reads: $\psi\in\mcn(\rn)$ has the volume doubling property if, and only if, $\liminf_{r\to\infty} \psi_*(Cr)/\psi_*(r) > 1$ and $\liminf_{r\to 0} \psi_*(Cr)/\psi_*(r) > 1$ for some $C>1$. If this is the case, then $\psi_*(r)$ grows, as $r\to\infty$, at least like a (fractional) power.

\begin{corollary}\label{cndf-25}
    Assume that $\psi\in\mcn(\rn)$ enjoys the volume doubling property and that $f$ is a Bernstein function such that
    $$
        \liminf_{r\to 0} \frac{f(Cr)}{f(r)}>1
        \quad\text{and}\quad
        \liminf_{r\to\infty} \frac{f(Cr)}{f(r)}>1
    $$
    for some $C>1$. Then $f\circ\psi\in\mcn$ and $f\circ\psi$ has the volume doubling property.
\end{corollary}
\begin{proof}
    If $f\circ\psi$ is volume doubling, then $\liminf_{|\xi|\to\infty} |f(\psi(\xi))| > 0$ and we get that $f\circ\psi\in\mcn(\rn)$, cf.\ Lemma \ref{cndf-05}.

    Assume that $\psi$ is volume doubling. Since $f$ is continuous, our assumptions ensure that $f^{-1}(C r)\leq \gamma f^{-1}(r)$ for all $r>0$ and some $\gamma>1$.
    Then
    \begin{align*}
        \lambda\left(B^{d_{f\circ\psi}}\big(0,\sqrt C r\big)\right)
        &= \lambda\left(B^{d_\psi}\big(0,\sqrt{f^{-1}( C r^2)}\big)\right)\\
        &\leq \lambda\left(B^{d_\psi}\big(0,\sqrt{\gamma} \sqrt{f^{-1}(r^2)}\big)\right)\\
        &\leq c_{\sqrt{\gamma}} \lambda\left(B^{d_\psi}\big(0,\sqrt{f^{-1}(r^2)}\big)\right)\\
        &=  c_{\sqrt{\gamma}} \lambda\left(B^{d_{f\circ\psi}}(0,r)\right).
    \end{align*}
    Because of Lemma \ref{cndf-05} we see that $f\circ\psi$ is volume doubling.
\end{proof}

\section{Understanding the role of $p_t(0)$}\label{pto}

Let $\psi\in\mcn(\rn)$ and denote by $(\mu_t)_{t\geq 0}$ the corresponding convolution semigroup satisfying $\Ff^{-1}\mu_t = e^{-t\psi}$. We assume that $e^{-t\psi}\in L^1(\rn,\lambda)$, so that the measures $\mu_t$ are absolutely continuous with respect to Lebesgue measure. The probability densities are given by
\begin{equation*}
    p_t(x) = (2\pi)^{-n}\int_{\rn} e^{-ix\xi} e^{-t\psi(\xi)}\,d\xi = \Ff e^{-t\psi}(x),\quad t>0,
\end{equation*}
and, by the Riemann-Lebesgue lemma, we know that $x\mapsto p_t(x)$ is a continuous function vanishing at infinity.

In this section we will discuss the relation of $p_t(0)$ with the geometry induced by the metric measure space $(\rn,d_\psi,\lambda)$. \textbf{If not stated otherwise, we will assume that $\psi$ is real-valued.} This means, in particular, that $\psi(\xi)\geq 0$ and that $p_t(\cdot)$ is an even function.

Since $(2\pi)^n \, p_t(0) = \int_{\rn} e^{-t\psi(\xi)}\,d\xi$ and $d_\psi(\xi,0) = \sqrt{\psi(\xi)}$, we see that
$$
    \sigma_t(d\xi) := \frac{e^{-td_\psi^2(\xi,0)}}{(2\pi)^n\,p_t(0)}\,d\xi,\quad t>0,
$$
are probability measures on $\rn$. Recall that $B^{d_\psi}(\eta,r) = \left\{\xi\in\rn \::\: d_\psi(\xi,\eta) < r \right\}$. The following result is essentially contained in \cite{KS1}. Note that in \cite{KS1} the condition $e^{-t\psi}\in L^1(\rn)$ is substituted by a more general (Hartman-Wintner type) condition on the growth of $\psi$.
\begin{theorem}\label{pto-03}
    Let $\psi\in\mcn(\rn)$ and assume that $e^{-t\psi}\in L^1(\rn,\lambda)$. Then
    \begin{equation}\label{pto-e04}
        p_t(0) = (2\pi)^{-n} \int_0^\infty \lambda \left(B^{d_\psi}\big(0,\sqrt{r/t}\big)\right)\,e^{-r}\,dr,\quad t>0.
    \end{equation}
    If the metric measure space $(\rn,d_\psi,\lambda)$ has the volume doubling property, then $e^{-t\psi}\in L^1(\rn,\lambda)$ and
    \begin{equation}\label{pto-e06}
        p_t(0) \asymp
        \lambda\left(B^{d_\psi}\big(0,1/\sqrt t\big)\right)
        \quad\text{for all}\quad t>0.
    \end{equation}
\end{theorem}
\begin{proof}
    Observe that $\lambda\left(B^{d_\psi}\big(0,\sqrt{\rho}\big)\right) = \lambda\left\{\xi\in\rn\::\: \psi(\xi) < \rho\right\}$ is the distribution function of $\psi$. Using Fubini's theorem we get
    \begin{align*}
        (2\pi)^n\,p_t(0)
        &= \int_\rn e^{-t\psi(\xi)}\,d\xi\\
        &= t\int_0^\infty \lambda\left(B^{d_\psi}\big(0,\sqrt{\rho}\big)\right)\,e^{-t\rho}\,d\rho\\
        &= \int_0^\infty \lambda\left(B^{d_\psi}\big(0,\sqrt{r/t}\big)\right)\,e^{-r}\,dr,
    \end{align*}
    and \eqref{pto-e04} follows.

    We know from Corollary \ref{cndf-23} that volume doubling implies power growth of the increasing rearrangement $\psi_* = v_\psi^{-1}$. By `d)$\Rightarrow$a)' of \cite[Proposition 5]{KS1} we get that $e^{-t\psi}\in L^1(\rn,\lambda)$.
    Using \eqref{pto-e04} and the monotonicity of the function $r\mapsto \lambda\left(B^{d_\psi}(0,r)\right)$ we get
    \begin{align*}
        (2\pi)^n\,p_t(0)
        &\geq \int_1^\infty \lambda\left(B^{d_\psi}\big(0,\sqrt{r/t}\big)\right) e^{-r}\,dr\\
        &\geq \lambda\left(B^{d_\psi}\big(0,1/\sqrt{t}\big)\right) \int_1^\infty e^{-r}\,dr\\
        &= \frac 1e\,\lambda\left(B^{d_\psi}\big(0,1/\sqrt{t}\big)\right).
    \end{align*}
    This proves the first inequality of \eqref{pto-e06}. The upper estimate requires that $\psi$ enjoys the volume doubling property. This means that
    $$
        \lambda\left(B^{d_\psi}(0,cr)\right)
        \leq \gamma_0(c)\,\lambda\left(B^{d_\psi}(0,r)\right), \quad c>1,\; r>0,
    $$
    for some function $\gamma_0$ such that $\gamma_0(c)\leq \gamma_0(1)\, c^\alpha$ for all $c\geq 1$ with some suitable constant $\alpha\geq 0$. Combining this with \eqref{pto-e04} gives
    \begin{align*}
        (2\pi)^n\,p_t(0)
        &= \int_0^1 \lambda\left(B^{d_\psi}\big(0,\sqrt{r/t}\big)\right) e^{-r}\,dr
          + \int_1^\infty \lambda\left(B^{d_\psi}\big(0,\sqrt{r/t}\big)\right) e^{-r}\,dr\\
        &\leq \left(1-e^{-1}\right) \lambda\left(B^{d_\psi}\big(0,1/\sqrt{t}\big)\right)
          + \lambda\left(B^{d_\psi}\big(0,1/\sqrt t\big)\right) \int_1^\infty \gamma_0(1) r^{\alpha/2}\,e^{-r}\,dr\\
        &= \kappa_1\,\lambda\left(B^{d_\psi}\big(0,1/\sqrt{t}\big)\right);
    \end{align*}
    this is the upper estimate in \eqref{pto-e06}.
\end{proof}

Let us illustrate Theorem \ref{pto-03} with several examples. First, however, we note that the estimate \eqref{pto-e06} indicates some kind of `Gaussian' behaviour of $p_t(0)$---it is comparable to the Lebesgue volume of a metric ball. This is exactly what we see in the Gaussian setting where $\psi(\xi)=\frac 12 |\xi|^2$ and
$
    p_t(x) = (2\pi t)^{-n/2}\,e^{-\frac{|x|^2}{2t}}.
$
Clearly,
$$
    p_t(0) = (2\pi t)^{-n/2}
    \quad\text{hence}\quad
    p_t(0) = c_n\,\lambda\left(B\big(0,1/\sqrt t\big)\right)
$$
where $B(0,r)$ stands for the usual ball in the Euclidean topology.
\begin{remark}\label{pto-05}
    Let $\psi\in\mcn(\rn)$ and denote by $p_t$ the corresponding transition density function (which we assume to exist). Let $(\eta_t)_{t\geq 0}$ be a convolution semigroup of measures on the half-line $[0,\infty)$ and denote by $f$ the corresponding Bernstein function $f$. The subordinate density $p_t^f$ is given by
    \begin{equation*}
        p_t^f(x) = \int_0^\infty p_s(x)\,\eta_t(ds),\quad t>0,
    \end{equation*}
    and, consequently, $p_t^f(0)=\int_0^\infty p_s(0)\eta_t(ds)$. By Theorem \ref{pto-03} we see
    $$
       p_t^f(0)
       \asymp \int_0^\infty \lambda\left(B^{d_\psi}\big(0,1/\sqrt{s}\big)\right)\eta_t(ds).
    $$

    Moreover, if $f\circ\psi$ has the volume doubling property, see e.g.\ Corollary \ref{cndf-25}, \eqref{pto-e06} gives
    \begin{equation*}
        \frac{p_t^f(0)}{\lambda\left(B^{d_{f\circ \psi}}\big(0,1/\sqrt t\big)\vphantom{B^{d_\psi}\big(0,\sqrt{f^{-1}(1/t)}\big)}\right)}
        = \frac{p_t^f(0)}{\lambda\left(B^{d_\psi}\big(0,\sqrt{f^{-1}(1/t)}\big)\right)}
        \asymp 1
    \end{equation*}
    and so
    \begin{equation*}
        \lambda\left(B^{d_{f\circ \psi}}\big(0,1/\sqrt t\big)\right)
        \asymp \int_0^\infty \lambda\left(B^{d_\psi}\big(0,1/\sqrt{s}\big)\right)\eta_t(ds).
    \end{equation*}
\end{remark}

\bigskip
\begin{example}\label{pto-07}
    Let $f$ be a Bernstein function. By Lemma \ref{cndf-05} we know that $\psi := f(|\cdot |^2)$ is in $\mcn(\rn)$, and from Proposition \ref{cndf-15} it follows that
    \begin{equation*}
        \lambda\left(B^{d_{f(|\cdot|^2)}}(0,cr)\right)
        = \left(\frac{f^{-1}(c^2 r^2)}{f^{-1}(r^2)}\right)^{n/2} \lambda\left(B^{d_{f(|\cdot|^2)}}(0,r)\right).
    \end{equation*}
    In particular, we get for $f(s)=s^\alpha$, $0<\alpha<1$, that
    $$
        \lambda\left(B^{d_{|\cdot|^{2\alpha}}}(0,cr)\right)
        = c^{n/\alpha} \lambda\left(B^{d_{|\cdot|^{2\alpha}}}(0,r)\right).
    $$
    Since $\lambda\left(B^{d_{|\cdot|^{2\alpha}}}(0,r)\right) = c_{n,\alpha}\,r^{n/\alpha}$ we recover the well-known estimates
    $$
        p_t^{(2\alpha)}(0) \asymp t^{-n/2\alpha}
    $$
    where $p_t^{(2\alpha)}(x)$ is the transition density of the symmetric $2\alpha$-stable L\'evy process. In fact, in this special case, we can calculate $p_t^{(2\alpha)}(0)$ exactly:
%
        $$
            p_t^{(2\alpha)}(0)
            = \alpha \Gamma\left(\tfrac n\alpha + 1\right) \, \lambda\left(B^{d_{|\cdot|^{2\alpha}}}\big(0,1/\sqrt t\big)\right).
        $$

    More generally, if $\psi\in\mcn$ then $f\circ\psi\in\mcn$ by Lemma \ref{cndf-05}. If $f\circ\psi$ has the doubling property, we get that $e^{-f\circ\psi}\in L^1(\rn,\lambda)$, cf.\ Theorem \ref{pto-03}, and
    \begin{equation}\label{pto-e22}
        p_t(0)
        \asymp
        \lambda\left(B^{d_\psi}\big(0,\sqrt{f^{-1}(1/t)}\big)\right)
    \end{equation}
    where $p_t(x) = p_t^{f\circ\psi}(x) = (2\pi)^{-n} \int e^{-ix\xi} e^{-t f(\psi(\xi))}\,d\xi$ is the transition density of the subordinate L\'evy process.
\end{example}

\bigskip
\begin{example}\label{pto-09}
    Consider on $\real^m\times\real^n$ the function $\psi(\xi,\eta)=|\xi|^\alpha + |\eta|^\beta$ with $0<\alpha<\beta<2$. Then $\psi\in\mcn(\real^m\times\real^n)$. It is shown in \cite{SL} that
    \begin{equation*}
        \lambda\left(B^{d_\psi}(0,R)\right)
        = \left(\frac Rr\right)^{2\left(\frac m\alpha + \frac n\beta\right)}
        \lambda\left(B^{d_\psi}(0,r)\right).
    \end{equation*}
    Consequently, we get for $p_t^\psi(0) = (2\pi)^{-n-m}\iint_{\Rm\times\rn} e^{-t(|\xi|^\alpha+|\eta|^\beta)}\,d\xi\,d\eta$ that
    $$
        p_t^\psi(0)
        \asymp
        t^{-\frac m\alpha - \frac n\beta}.
    $$
    Note that \eqref{pto-e22} controls both the growth of the singularity of $p_t^\psi(0)$ as $t\to 0$ and the decay of $p_t^\psi(0)$ as $t\to\infty$. Both controls are related to volume growth in the corresponding metric measure spaces. A further consequence of \eqref{pto-e22} is that if any two $\psi_1,\psi_2\in\mcn(\rn)$ are comparable in the sense that
    \begin{equation*}
        \tau_0\,d_{\psi_1}(\xi,\eta)
        \leq \,d_{\psi_2}(\xi,\eta)
        \leq \tau_1\,d_{\psi_1}(\xi,\eta),\quad \xi,\eta\in\rn,
    \end{equation*}
    for suitable constants $0<\tau_0\leq \tau_1<\infty$, then $p_t^{\psi_1}(0)$ and $p_t^{\psi_2}(0)$ have comparable growth behaviour as $t\to 0$ and $t\to\infty$. Nevertheless, we cannot expect any obvious comparison of $p_t^{\psi_1}(\xi)$ and $p_t^{\psi_2}(\xi)$ if $\xi\neq 0$.

    Take, for example, $\psi_1(\xi,\eta)=|\xi|+|\eta|$ and $\psi_2(\xi,\eta) = \sqrt{\xi^2+\eta^2}$ where $\xi,\eta\in\real$. Then
    $$
        \frac 1{\sqrt 2}\,(|\xi|+|\eta|)
        \leq \sqrt{|\xi|^2+|\eta|^2}
        \leq |\xi|+|\eta|.
    $$
    On the other hand, $\psi_1$ and $\psi_2$ have different smoothness properties near the origin 
    and it is this property that determines the decay of the transition functions $p_t^{\psi_1}(x,y)$ and $p_t^{\psi_2}(x,y)$ as $|x|+|y|\to\infty$. In fact, we have
    $$
        p_t^{\psi_1}(x,y) = \frac 1{\pi^2}\,\frac{t^2}{(x^2+t^2)(y^2+t^2)}
    $$
    and
    $$
        p_t^{\psi_2}(x,y) = \frac{1}{2\pi}\,\frac{t}{\big((x^2+y^2)+t^2\big)^{3/2}}
    $$
    which gives, for example, for $|x|\to\infty$,
    $$
        p_1^{\psi_1}(x,0) \asymp |x|^{-2}
        \quad\text{while}\quad
        p_1^{\psi_2}(x,0) \asymp |x|^{-3}.
    $$
\end{example}

\section{On the off-diagonal behaviour of $p_t(x)$}\label{ptx}

As in Section \ref{pto} we assume that $\psi\in\mcn(\rn)$ is real valued and that the associated convolution semigroup $(\mu_t)_{t>0}$ is absolutely continuous with respect to Lebesgue measure, $\mu_t(dx) = p_t(x)\,dx$. We have seen that $p_t(0)$ has a natural meaning in the metric measure space $(\rn,d_\psi,\lambda)$. We have
\begin{equation}\label{ptx-e04}
    \frac{p_t(x)}{p_t(0)}
    = \int_{\rn} e^{-ix\xi} \, \frac{e^{-t d_\psi^2(\xi,0)}}{(2\pi)^{n}\,p_t(0)}\,d\xi
    = \int_{\rn} e^{-ix\xi} \, \frac{e^{-t \psi(\xi)}}{(2\pi)^{n}\,p_t(0)}\,d\xi,
\end{equation}
and we want to understand $p_t(x)/p_t(0)$ better. For this we start with a few examples.
\begin{example}\label{ptx-03}
\textbf{a)} Let $\psi(\xi)=\frac 12 |\xi|^2$. Then $p_t(x)$ is the Gauss kernel in $\rn$,
\begin{equation*}
    \frac{p_t(x)}{p_t(0)} = \exp\left(-\frac{|x|^2}{2t}\right).
\end{equation*}
and for the (almost Euclidean) metric $\phi_t(x,y) = |x-y|/\sqrt{2 t}$ we get
\begin{equation*}
    \frac{p_t(x)}{p_t(0)} = \exp\left(-\phi_t^2(x,0)\right).
\end{equation*}

\bigskip\noindent\textbf{b)} Let $\psi(\xi)=|\xi|$. Then $p_t(x)$ is the density of the Cauchy process in $\rn$ and we find
\begin{equation*}
    \frac{p_t(x)}{p_t(0)} = \exp\left(-\phi_t^2(x,0)\right)
\end{equation*}
with
\begin{equation*}
    \phi_t(x,y) = \sqrt{ \frac{n+1}2\,\ln\left[\frac{|x-y|^2+t^2}{t^2}\right]}.
\end{equation*}
If we fix $t>0$, then $\phi_t(\cdot,\cdot):\rn\times\rn\to\real$ is a translation invariant metric. Symmetry, positivity and definiteness are obvious. The triangle inequality follows from the fact that $r\mapsto\ln\left(1+\frac ra\right)$ is a Bernstein function and that $\phi_t^2(x,0)$ is a continuous negative definite function.
\end{example}

Thus we are naturally led to the following question.
\begin{problem}\label{ptx-05}
    Let $p_t$ be the transition density of some symmetric L\'evy process with characteristic exponent $\psi\in\mcn(\rn)$. Does there exist a mapping $\delta_\psi:(0,\infty)\times\rn\times\rn\to\real$ such that for every $t\in (0,\infty)$ the map $\delta_{\psi,t}(\cdot,\cdot):\rn\times\rn\to\real$ is a (translation invariant) metric such that
    \begin{equation}\label{ptx-e14}
        \frac{p_t(x)}{p_t(0)} = \exp\left(-\delta_{\psi,t}^2( x,0)\right)\;?
    \end{equation}
\end{problem}

Below we will see many more concrete examples for which Problem \ref{ptx-05} can be answered in the affirmative. There is, however, a general result supporting the conjecture that \eqref{ptx-e14} should hold for all symmetric L\'evy processes with characteristic exponent $\psi\in\mcn(\rn)$.

The key observation is that the metric space $(\rn,d_\psi)$, $\psi\in\mcn(\rn)$, can be isometrically embedded into some Hilbert space. The following result is originally due to Schoenberg \cite{Schoe1,Schoe2} and it led to the notion of negative definite functions; a modern account is given in the monograph \cite{BL} by Benyamini and Lindenstrauss.
\begin{theorem}[Schoenberg]\label{ptx-06}
    Let $\psi\in\mcn(\rn)$. Then the metric measure space $(\rn,d_\psi)$ can be isometrically embedded into some Hilbert space $(\mathcal H,\scalar{\cdot,\cdot}_{\mathcal H})$.

    Conversely, if $J:(\rn,d)\to(\mathcal H,\scalar{\cdot,\cdot}_{\mathcal H})$ is an isometric embedding into some Hilbert space such that $\scalar{J(x),J(y)}_{\mathcal H} = \scalar{J(x-y),J(0)}_{\mathcal H}$, then $d=d_\psi$ for some negative definite function $\psi:\rn\to\real$.
\end{theorem}

Let us give a \emph{sketch of the proof} of the sufficiency using the theory of Dirichlet forms, see \cite{FOT} or \cite{J1}. Our proof will also reveal the structure of the embedding. In order to exclude trivial cases, we will assume that $\psi$ has no quadratic part. Therefore, the L\'evy-Khintchine formula \eqref{aux-e13} becomes
$$
    \psi(\xi)
    = \int_{y\neq 0} \left(1-\cos(y\xi)\right)\nu(dy).
$$

We can associate with every continuous negative definite function $\psi:\rn\to\real$ a pseudo-differential operator $\psi(D)$ on $\mathscr S(\rn)$ defined by
\begin{equation*}
    \psi(D)u(x)
    = \int_\rn e^{ix\xi}\psi(\xi)\,\Ff u(\xi)\,d\xi.
\end{equation*}
Introducing the scale of Hilbert spaces $H^{\psi,s}(\rn)$, $s\geq 0$,
\begin{equation*}
\begin{gathered}
    H^{\psi,s}(\rn) := \left\{u\in L^2(\rn) \::\: \nnorm{u}_{\psi,s}<\infty\right\},\\
    \nnorm{u}_{\psi,s}^2 := \int_\rn (1+\psi(\xi))^s |\Ff u(\xi)|^2\,d\xi,
\end{gathered}
\end{equation*}
it is easy to see that $\psi(D):H^{\psi,s+2}(\rn)\to H^{\psi,s}(\rn)$ is continuous and that the quadratic form $\mathcal E^\psi$ associated with $\psi(D)$ by
\begin{equation*}
    \mathcal E^\psi(u,v)
    = \int_\rn \psi(D)^{1/2}u(x) \psi(D)^{1/2}v(x)\,dx
    = (2\pi)^n \int_\rn \psi(\xi)\,\Ff u(\xi) \, \overline{\Ff v(\xi)}\,d\xi
\end{equation*}
is closed on $D(\mathcal E^\psi) = H^{\psi,1}(\rn)$. Using the L\'evy-Khintchine representation we get
\begin{equation*}
    \mathcal E^\psi(u,v)
    = \frac 12 \iint_{\rn\times\rn\setminus\{0\}} \big(u(x+y)-u(x)\big)\big(v(x+y)-v(x)\big)\,\nu(dy)\,dx.
\end{equation*}
We call the operator $(u,v)\mapsto\Gamma(u,v)$ where
\begin{equation*}
    \Gamma(u,v)(x)
    = \frac 12 \int_{\rn\setminus\{0\}} \big(u(x+y)-u(x)\big)\big(v(x+y)-v(x)\big)\,\nu(dy)
\end{equation*}
the \emph{carr\'e du champ operator} associated with the Dirichlet form $\mathcal E^\psi$. For a comprehensive discussion of the carr\'e du champ operator we refer to \cite{BH}.

Consider the set
\begin{equation*}
    \mathcal C_\psi
    := \big\{ u\in C^2(\rn)\::\: \Gamma(u,u)(0)<\infty\big\} \Big/ \big\{ u\in C^2(\rn)\::\: \Gamma(u,u)(0)=0\big\}
\end{equation*}
and notice that all constants belong to the set $\big\{ u\in C^2(\rn)\::\: \Gamma(u,u)(0)=0\big\}$.

If the L\'evy measure of $\psi$ has full support, $\supp\nu = \rn$---e.g.\ if $\nu = g\lambda$ with an everywhere strictly positive density $g>0$---, then $\Gamma(u,v)(0)$ is a scalar product on $\mathcal C_\psi$. Let us denote this scalar product, for a moment, by
$$
    \scalar{u,v}_{\mathcal H} := \Gamma(u,v)(0).
$$
Then $\mathcal C_\psi = \big\{ u\in C^2(\rn)\::\: \Gamma(u,u)(0)<\infty\big\} \big/ \big\{ u\equiv\text{const}\big\}$. (Since $\nu$ has full support it is not hard to see that $\big\{ u\in C^2(\rn)\::\: \Gamma(u,u)(0)=0\big\}$ are exactly the constant functions). The completion of $\mathcal C_\psi$ with respect to  $\scalar{\cdot,\cdot}_{\mathcal H}$ gives a Hilbert space $(\mathcal H,\scalar{\cdot,\cdot}_{\mathcal H})$. Each function $e_\xi$, $e_\xi(x):= e^{ix\xi}$, $\xi\in\rn$, represents some element of $\mathcal H$ and we get
\begin{equation*}
    \Gamma(e_\xi,e_\xi)(0) = \psi(\xi).
\end{equation*}
This shows that the map $J:\rn\to\mathcal H$, $\xi\mapsto e_\xi$, embeds the metric space $(\rn,d_\psi)$ isometrically into the Hilbert space $(\mathcal H, \scalar{\cdot,\cdot}_{\mathcal H})$.

\bigskip
Thus, we can understand $e^{-t\psi(\xi)}$ as a \emph{Gaussian in disguise}: Fourier transforms of Gaussians should be Gaussians, and Gaussians have obviously the proposed structure. It is tempting to find a representation of $p_t(x)$ as some image of an infinite dimensional Gaussian defined somehow on $(\mathcal H,\scalar{\cdot,\cdot}_{\mathcal H})$. So far, however, such a result resists all of our attempts to prove it.

Before we return to Problem \ref{ptx-05} and before we provide more examples, we want to give an interpretation of \eqref{ptx-e04}.

We have
\begin{equation*}
    p_t(x) = p_t(0) \exp\left(-\delta_{\psi,t}^2( x,0))\right)
\end{equation*}
and if $\psi$ satisfies the volume doubling condition, we get from Theorem \ref{pto-03}
\begin{equation*}
    p_t(x)  \asymp
    \lambda\big(B^{d_\psi}(0,1/\sqrt t)\big) \exp\big(-\delta_{\psi,t}^2( x,0))\big)
\end{equation*}
with  the balls $B^{d_\psi}(0,r)=\left\{y\in\rn \::\: d_\psi(y,0)<r \right\}$. Thus, if $(\rn,d_\psi,\lambda)$ has the volume doubling property, $p_t(x)$ is controlled by two geometric expressions. More precisely, for fixed $t>0$ we have two metrics $\delta_{\psi,t}(\cdot,\cdot)$ and $d_\psi(\cdot,\cdot)$ which describe the behaviour of $p_t(x)$. The situation becomes more transparent if we switch from $d_\psi(x,y)=\sqrt{\psi(x-y)}$ to $d_{\psi,t} : \rn\times\rn\to\real$, $t>0$, $d_{\psi,t}(x,y)= \sqrt{t\cdot\psi(x-y)}$ since, in this new metric,
\begin{align*}
    B^{d_\psi}\big( 0,\tfrac 1{\sqrt t}\big)
    &= \left\{ y\in\rn \::\: d_\psi(y,0)<\tfrac 1{\sqrt t}\right\}\\
    &= \left\{ y\in\rn \::\: d_{\psi,t}( y,0)<1\right\}
    =: B^{d_{\psi,t}}(0,1).
\end{align*}
From now on we will adopt the point of view that $p_t(x)$ should be understood in terms of two families of metrics, $d_{\psi,t}( \cdot, \cdot)$ and $\delta_{\psi,t}(\cdot,\cdot)$, by the estimates
\begin{equation*}
    p_t(x) \asymp \lambda\left(B^{d_{\psi,t}}(0,1)\right) \exp\left(-\delta_{\psi,t}^2(x,0)\right).
\end{equation*}
Consequently, the understanding of $p_t(x)$ is reduced to the study of $d_{\psi,t}(\cdot,\cdot)$ and $\delta_{\psi,t}(\cdot,\cdot)$. For a Brownian motion $\psi(\xi)=\frac 12|\xi|^2$; in this case both metrics $\delta_{\psi,t}$ and $d_\psi$ are (essentially) Euclidean distances. This special situation is related to the fact that the Gaussian is, up to constants, a fixed point of the Fourier transform.

This interpretation allows us also to think about distributions of random variables or collections of densities $p_t$, $t\in I\subset (0,\infty)$. We will follow up this remark in the next section.

Of special interest is the case when \eqref{ptx-e14} holds with a metric $\delta_{\psi,t}$ such that $x\mapsto \delta_{\psi,t}^2(x,0)$ is a negative definite function. This is, e.g.\ the case in Example \ref{ptx-03}. Here is a further example.
\begin{example}\label{ptx-07}
    The symmetric Meixner process on $\real$ has the characteristic L\'evy exponent $\psi(\xi) = \ln(\cosh \xi)$ and the transition density
    \begin{equation*}
        p_t(x) = \frac{2^{t-1}}{\pi\,\Gamma(t)}\,\left|\Gamma\left(\frac{t+ix}{2}\right)\right|^2.
    \end{equation*}
    Using the representation of the Gamma function as an infinite product, we find
    \begin{equation*}
        \frac{p_t(x)}{p_t(0)}
        = \left|\frac{\Gamma\left(\frac{t+ix}{2}\right)}{\Gamma\left(\frac{t}{2}\right)}\right|^2
        = \prod_{j=1}^\infty \left(1+\frac{x^2}{(t+2j)^2}\right)^{-1}.
    \end{equation*}
    Since for a sequence $(a_j)_{j\geq 1}$ of positive numbers the convergence of
    $$
        \prod_{j=1}^\infty (1+a_j),\quad
        \sum_{j=1}^\infty \ln(1+a_j)\quad\text{and}\quad
        \sum_{j=1}^\infty a_j
    $$
    is equivalent, we find that
    \begin{equation*}
        \delta^2_t(x,0)
        = - \ln \left|\frac{\Gamma\left(\frac{t+ix}{2}\right)}{\Gamma\left(\frac{t}{2}\right)}\right|^2
        = \sum_{j=1}^\infty \ln\left(1+\frac{x^2}{(t+2j)^2}\right).
    \end{equation*}
    For every $j\geq 1$ and $t>0$ the function $x\mapsto \ln\big(1+x^2(t+2j)^{-2}\big)$ is continuous and negative definite. Since the series $\sum_{j=1}^\infty \ln\big(1+x^2(t+2j)^{-2}\big)$ converges locally uniformly as a function of $x$, its sum $\delta^2_t(x,0)$ is again a continuous negative definite function. This shows that the transition density of a symmetric Meixner process satisfies
    \begin{equation}\label{ptx-e36}
        p_t(x) = p_t(0) \, \exp\left(-\delta^2_t(x,0)\right)
    \end{equation}
    and $p_t(0)$ can be written, as before, as
    \begin{equation*}
        p_t(0)
        = (2\pi)^{-n}\int_\rn e^{-\frac t2 \ln\left(\cosh^2\xi\right)}\,d\xi
        = (2\pi)^{-n}\int_\rn e^{- t \ln\left(\cosh\xi\right)}\,d\xi.
    \end{equation*}

    Although $x\mapsto\delta^2_t(x,0)$ is a continuous negative definite function, $e^{-\delta_t^2(x,0)}$ is \emph{not} the characteristic function of an additive process, i.e.\ a stochastically continuous process with independent, but not necessarily stationary, increments, see \cite{Sato}. This can be seen from the the L\'evy-Khintchine representation for $\delta_t^2(x,0)$:
\begin{align*}
    \delta_t^2(x,0)
    &=\sum_{j=0}^\infty \ln \left(1+\frac{x^2}{(t+2j)^2}\right)\\
    &=\int_{\real\setminus\{0\}} \frac{1-\cos(xz)}{|z|} \left( \sum_{j=0}^\infty e^{-|z|(t+2j)}\right) dz\\
    &=\int_{\real\setminus\{0\}} \bigl(1-\cos(xz)\bigr)\,g(t,z)\,dz
\end{align*}
    with $g(t,z):=\sum_{j=0}^\infty e^{-|z|(t+2j)}$. In this calculation we used the L\'evy-Khintchine representation for the continuous negative definite function
$$
    \ln \left(1+\frac{\xi^2}{a^2}\right)
    =\int_{\real\setminus\{0\}} \bigl(1-\cos(\xi z)\bigr)\,\frac{e^{-|z|a }}{|z|}\,dz.
$$
    Since  $g(t,z)$ is decreasing in $t$, we see that $e^{-\delta_t^2(x,0)}$ cannot be the characteristic function of an additive process.
\end{example}

In the next  Sections  \ref{dist} and \ref{proc} we will continue our investigation of \eqref{ptx-e36}. In fact, we will encounter a large class of processes with transition function
$$
    p_t(x) = p_t(0) e^{-g(t;|x|^2)}
$$
where for every fixed $t>0$ the function $g(t;\cdot)$ is a Bernstein function.

We want to collect some more information on $p_t$ and $p_t/p_t(0)$. Since $p_t$ is the Fourier transform of a measure, it is a positive definite function. The measure $\rho_t$ defined by
\begin{equation*}
    \rho_t(dx) := \frac{\Ff^{-1}p_t(x)}{p_t(0)}\,dx
\end{equation*}
is a probability measure and if $X_t$ is a random variable with distribution $\rho_t$, we have
$$
    \Pp(X_t\in dx) = \rho_t(dx) = \frac{e^{-t\psi(x)}}{p_t(0)}\,dx.
$$
Thus, we have a certain duality. Given a (symmetric) L\'evy process $(Y_t^\psi)_{t\geq 0}$ with characteristic exponent $\psi\in\mcn(\rn)$, then $Y_t^\psi$ is for every $t>0$ associated with $X_t$ and vice versa. Of particular interest should be the case where $(X_t)_{t\geq 0}$ is itself a `nice' process, say an additive process or even a L\'evy process. In Section \ref{proc} we will provide some examples for such pairings.

Our starting point was to understand the densities of L\'evy processes in terms of two one-parameter families of metrics, $(d_{\psi,t}(\cdot,\cdot))_{t>0}$ and $(\delta_{\psi,t}(\cdot,\cdot))_{t>0}$. Although this is still the main aim of our study, it is often useful to fix $t=t_0$ and to consider a single probability density $p(x) = p_{t_0}(x)$ rather than the whole family $(p_t(x))_{t>0}$. This is no loss of generality since we can embed every infinitely divisible probability density $p(x)$ into a convolution semigroup $(\mu_t)_{t\geq 0}$ such that $\mu_{t_0}(dx) = p(x)\,dx$. In the following section and in Section \ref{proc} we will study examples of infinitely divisible probability densities.

\section{Examples of class $\Nsf$ distributions}\label{dist}

The following definition covers all cases mentioned in Problem \ref{ptx-05} where the exponent $\delta_{\psi,1}(x,0)$ appearing in \eqref{ptx-e14} is not only a metric but a metric induced by some negative definite function.
\begin{definition}\label{dist-03}
    By $\Nsf$ we denote the class of infinitely divisible probability densities $p(x)$ on $\rn$, such that $\Ff^{-1} p(\xi)/p(0)$ is again an infinitely divisible probability density.
\end{definition}
Note that $\Ff^{-1} p(\xi) = e^{-\phi(\xi)}$ is always an \emph{infinitely divisible characteristic function} with a continuous negative definite characteristic exponent $\phi:\rn\to\comp$.  This is different from saying that $\Ff^{-1} p(\xi)/p(0)$ is an \emph{infinitely divisible probability distribution}; $p\in \Nsf$ means that $-\ln \bigl[p(x)/p(0)\bigr]$ is a continuous negative definite function. In particular, all distributions in the class $\Nsf$ provide solutions to Problem \ref{ptx-05}.

In order to get examples of class $\Nsf$ distributions we start with a simple class of examples related to a single infinitely divisible random variable $X$ on $\rn$ with an (infinitely divisible) probability density $p(x)$ which is symmetric. Then
$$
    \Ee e^{i\xi X}
    = \Ff^{-1} p(\xi)
    = e^{-\phi(\xi)}
$$
where $\phi:\rn\to\real$ is a continuous negative definite function. Suppose that $p$ is \emph{extended (Fourier) self-reciprocal} in the sense  that there is a constant $\gamma>0$ and a non-degenerate matrix $C\in\real^{n\times n}$, $\text{det}(C)>0$, such that
\begin{equation*}
    \Ff^{-1}p(\xi) = \frac 1\gamma\,p\big( C\xi\big)\quad\text{for all}\quad \xi\in\rn.
\end{equation*}
It follows that $p(0)=\gamma$ and
\begin{equation*}
    p(\xi) = p(0)\,e^{-d_\phi^2(C^{-1}\xi,0)}\quad\text{for all}\quad \xi\in\rn,
\end{equation*}
where $d_\phi(\xi,\eta) := \sqrt{\phi(\xi-\eta)}$ is the metric induced by the continuous negative function $\phi$.
\begin{example}\label{dist-05}
\textbf{a)}
    For every $\sigma>0$ the normal distribution $N(0,\sigma^2)$ is extended Fourier self-reciprocal and we have
    $$
        \frac 1{(2\pi\sigma^2)^{n/2}}\, \exp\left(-\frac{|\xi|^2}{2\sigma^2}\right)
        = \exp\left(-\frac 12\,\sigma^2\,|\xi|^2\right)
        \quad\text{for all}\quad \xi\in\rn.
    $$

\medskip\noindent\textbf{b)}
    The one-dimensional symmetric Meixner process $(X_t^M)_{t\geq 0}$ has $\psi(\xi) = \ln(\cosh\xi)$, $\xi\in\real$, as characteristic exponent. Denote by $(p_t(x)\,dx)_{t>0}$ the corresponding convolution semigroup.  For $t=1$ we  have, see \cite{PY03},
    $$
        p_1(x) = \frac{1}{2\cosh\big(\frac 12\,\pi x \big)}
        \quad\text{and}\quad
        \Ff^{-1} p_1(\xi) = \frac 1{\cosh \xi}.
    $$

\medskip\noindent\textbf{c)}
    Denote by $K_\lambda$ the modified Bessel function of the third kind, cf.\ \cite[vol.\ 2]{erd-et-al}. Let $Q>0$ and $\kappa>0$ be two positive numbers. Then
    $$
        p(x)
        := \frac{\kappa^{1/4}}{\sqrt{2\pi} \, K_{1/4}(\kappa^2)} \, \frac{K_{-1/4}\big(\kappa\sqrt{\kappa^2 + (x/Q)^2}\big) }{\big(\kappa^2 + (x/Q)^2\big)^{n/8}}
    $$
    is the density of a generalized hyperbolic distribution on $\real$ which is known to be infinitely divisible, see \cite{BNH} and \cite{BKS}. Then
    $$
        \Ff^{-1} p(\xi)
        = \frac{\kappa^{1/4}}{K_{1/4}(\kappa^2)} \, \frac{K_{1/4}\big(\kappa\sqrt{\kappa^2 + (Qx)^2}\big) }{\big(\kappa^2 + (Qx)^2\big)^{n/8}}.
    $$
    Since $K_{1/4} = K_{-1/4}$, we see that generalized hyperbolic densities are extended Fourier self-reciprocal.
\end{example}

Examples \ref{dist-05} a) and c) are special cases of so-called \emph{normal variance-mean mixtures}. The following definition is taken from \cite{BKS}.
\begin{definition}\label{dist-07}
    An $n$-dimensional random vector $Z$ is called a \emph{normal var\-iance-mean mixture},  if $Z$ is an $n$-dimensional normal distribution with covariance $Y Q$ and mean vector $\mu+Y \beta$, where $Q\in\real^{n\times n}$ is a symmetric positive definite matrix, $\mu$ and $\beta$ are $n$-dimensional matrices, and $Y$ is a positive random variable with probability law $\rho(ds)$ on $[0,\infty)$;  $\rho$ is called the \emph{mixing (probability) distribution}.
\end{definition}
Since we are mainly interested in the symmetric case, we assume from now that $\mu=\beta=0$. Denote by $p^\rho(dx)$ the law of $Z$; by definition it is the mixture, in the sense of \eqref{aux-e26}, of the Gaussian density $ p(s,x)=(2\pi s)^{-n/2} e^{-\frac{x \cdot Q^{-1} x}{2s}}$, $Q\in\real^{n\times n}$ is positive definite, and the probability measure $\rho(ds)$. If the  law $p^\rho(dx)$ is absolutely continuous with respect to $n$-dimensional Lebesgue measure, the density $p^\rho(x)$ is given by
\begin{equation}\label{dist-e06}
    p^\rho(x)=\int_0^\infty (2\pi s)^{-n/2} \exp\left[ -\frac{1}{2s}\, x\cdot Q^{-1}x\right] \rho(ds), \quad x\in\rn.
\end{equation}
Set $\phi(\lambda):=\mathcal{L} \rho(\lambda)$. Then we can calculate the characteristic function of $Z$ as a composition of the characteristic functions of a symmetric Gaussian distribution and $\phi$,
\begin{equation}\label{dist-e08}
    \Ff^{-1}[p^\rho](\xi)= \phi\left(-\frac{1}{2} \xi \cdot Q \xi\right).
\end{equation}
In general, we do not assume that the mixing probability measure $\rho$ is infinitely divisible.

The following theorem from \cite{BKS} gives sufficient conditions when normal variance-mean mixture is Fourier self-reciprocal. Its proof follows directly from the representation \eqref{dist-e06}.
\begin{theorem}\label{dist-09}
    Let $p^\rho(x)$ be a normal variance-mean mixture given by \eqref{dist-e06} and assume that the mixing probability measure $\rho(ds)=\rho(s)\,ds$ is absolutely continuous with respect to Lebesgue measure on $[0,\infty)$. If the density $\rho(s)$, $s\geq 0$, satisfies
\begin{equation}\label{dist-e09}
    \rho(s)=s^{(n-4)/2}\, \rho\left(\frac{1}{s}\right),
\end{equation}
    then $p^\rho=c \, \Ff^{-1}[p^\rho]$  with $c = p^\rho(0)$.
\end{theorem}

For our purposes the following simple corollary is important.
\begin{corollary}\label{dist-11}
    If the probability density $p^\rho(x)$, $x\in\rn$, is infinitely divisible and Fourier self-reciprocal, then  $p^\rho \in \Nsf$.
\end{corollary}

We have seen in Example \ref{dist-05}c) that the one-dimensional symmetric Meixner and the one-dimensional generalized hyperbolic distributions are (extended) Fourier self-reciprocal, hence of class $\Nsf$. Using the mixing result from Corollary \ref{dist-11} we can extend this result to the $n$-dimensional generalized hyperbolic distribution with parameters $\lambda=n/4$, $\beta=\mu=0$ and $\kappa=\eta$. Let us remark that for $\lambda=n/4$ and $\kappa\neq \eta$, although Theorem \ref{dist-09} does not cover this case, we are still in the generalized self-reciprocal setting. Indeed,
$$
    \rho_{\eta,\kappa,\lambda} (s)
    =\left(\frac{\kappa}{\eta}\right)^{2\lambda} s^{\frac{n-4}{2}} \rho_{\kappa,\eta,\lambda} (1/s).
$$

Let $Y\sim\rho$ be a mixing random variable where $\rho(ds):=\rho_{\eta,\kappa,\lambda}(s)\,ds$ is the generalized inverse Gaussian distribution
\begin{equation}\label{dist-e14}
    \rho_{\eta,\kappa,\lambda}(s)
    =\frac{(\kappa/\eta)^\lambda}{2K_\lambda (\eta \kappa)}\, s^{\lambda-1} \exp\left[-\frac{1}{2}\,\big(\eta^2 s^{-1} + \kappa^2 s\big)\right],
    \quad s>0,
\end{equation}
with parameters $\eta, \kappa$ and $\lambda$ satisfying
\begin{equation}\label{dist-e15}
\begin{cases}
    \eta\geq 0, \quad \kappa>0  &\text{if\ \ } \lambda>0;\\
    \eta>0, \quad \kappa>0      &\text{if\ \ } \lambda=0;\\
    \eta>0, \quad \kappa\geq 0  &\text{if\ \ } \lambda<0.
\end{cases}
\end{equation}
From \cite{BNH} we know that $\rho_{\eta,\kappa,\lambda}(s)$ is an infinitely divisible probability density on $(0,\infty)$.

Consider the probability density $p^\rho$ as in \eqref{dist-e06} with mixing probability $\rho = \rho_{\eta,\kappa,\lambda}$ as in \eqref{dist-e14}. The formulae for $p^\rho$ and $\Ff^{-1}{p^\rho}$ can be explicitly calculated, see e.g.\ \cite{BNH} and \cite{BKS},
\begin{gather}
    p^\rho(x)
    =\label{dist-e16}
    \frac{(\kappa/\eta)^{\lambda}}{(2\pi)^{n/2}K_\lambda(\eta \kappa)}\,
    \frac{K_{\lambda-n/2} \big(\kappa\sqrt{\eta^2+x\cdot Q^{-1}x}\big)}{\big(\kappa^{-1}\sqrt{\eta^2+x\cdot Q^{-1} x} \big)^{n/2-\lambda}},\quad x\in\rn,
\\
    \Ff^{-1}[p^\rho](\xi)
    =\label{dist-e18}
    \left(\frac{\kappa^2}{\kappa^2+\xi\cdot Q \xi}\right)^{\lambda/2}\,
    \frac{K_\lambda\big(\eta \sqrt{\kappa^2+\xi\cdot Q \xi}\big)}{K_\lambda(\eta \kappa)}, \quad \xi\in\rn.
\end{gather}
The probability distributions with the densities $p^\rho$ are called \emph{generalized hyperbolic distributions} with parameters $\eta$, $\kappa$ and $\lambda$ (satisfying \eqref{dist-e15}). Since $\rho = \rho_{\eta,\kappa,\lambda}$ is infinitely divisible, so is $p^\rho$; as  $K_\lambda = K_{-\lambda}$, it is easy to see that $p^\rho$ is extended (Fourier) self-reciprocal for $\lambda = n/4$, hence $p^\rho\in\Nsf$ by Corollary \ref{dist-11}.

For general $\lambda\geq 0$ and parameters $\eta,\kappa$ satisfying \eqref{dist-e15} we can use the fact that \eqref{dist-e16} and \eqref{dist-e18} have the same structure to conclude that $-\ln \bigl(p^\rho(x)\big/p^\rho(0)\bigr)$ is a continuous negative definite function. In the table in Section \ref{table} below we give several one-dimensional examples of such $p^\rho$. Our considerations show
\begin{theorem}\label{dist-13}
    Let $Q\in\real^{n\times n}$ be a positive definite matrix, let $\eta,\kappa,\lambda$ satisfy the conditions \eqref{dist-e15} and denote by $\rho = \rho_{\eta,\kappa,\lambda}$ the generalized inverse Gaussian distribution \eqref{dist-e14}. Then the generalized hyperbolic distribution $p^\rho$ from \eqref{dist-e16} has the property that
    $$
        \psi_{\eta,\kappa,\lambda}(x)
        :=
        -\ln \frac{p^\rho(x)}{p^\rho(0)}
        =
        -\ln\left(\frac{\eta^{n/2-\lambda}\, K_{\lambda-n/2} \big(\kappa \sqrt{\eta^2 + x\cdot Q^{-1}x} \big)}{K_{\lambda-n/2}(\kappa\eta) \big(\eta^2 + x\cdot Q^{-1}x \big)^{n/4 - \lambda/2}}\right),
        \quad x\in\rn,
    $$
    is a continuous negative definite function which induces a metric $\delta_{\eta,\kappa,\lambda}(x,y) = \sqrt{\psi_{\eta,\kappa,\lambda}(x-y)}$ of class $\mcn(\rn)$.   In particular, $p^\rho\in\Nsf$.
\end{theorem}
\begin{proof}
    We only have to show that $\delta_{\eta,\kappa,\lambda}\in\mcn(\rn)$. Since
    $$
        K_\nu(x) \approx \sqrt{\frac{\pi}{2|x|}} \, e^{-c|x|}
        \quad\text{as}\quad
        |x|\to \infty,
    $$
    cf.\ \cite[vol.\ 2, \S7.4.1]{erd-et-al}, we get  $\delta^2_{\eta,\kappa,\lambda}(x,0)= |x|+o(|x|)$ as $|x|\to\infty$. This implies that $\delta_{\eta,\kappa,\lambda}\in\mcn(\rn)$ by Lemma \ref{cndf-05}.
\end{proof}

Let us add one more example of class $\Nsf$ distributions obtained by mixing. Consider the probability density $p_2^S$ obtained by mixing \eqref{dist-e06} with the probability measure $\rho^S_2(ds)$ whose Laplace transform is
\begin{equation}\label{dist-e20}
    \mathcal{L}[\rho^S_2](\lambda)
    =\left( \frac{\sqrt{2\lambda}}{\sinh \sqrt{2\lambda}}\right)^2.
\end{equation}
The density $p_2^S$ and its characteristic function $\Ff^{-1}[p_2^S]$ can be calculated explicitly,
\begin{equation*}
    p_2^S(x)=\frac{\frac{\pi}{2} \left(\frac{\pi x}{2} \coth(\frac{\pi x}{2})-1\right)}{\sinh^2 (\frac{\pi x}{2})}
    \quad\text{and}\quad
    \Ff^{-1}[p_2^S](\xi)=\left(\frac{\xi}{\sinh\xi}\right)^2,
\end{equation*}
see \cite[p.\ 312, Table 6]{PY03}.  Note that $p_2^S$ is infinitely divisible because the mixing measure $\rho_2^S$ is infinitely divisible.
\begin{lemma}\label{dist-15}
    The transition probability density $p_2^S$ belongs to the class $\Nsf$.
\end{lemma}
The proof of Lemma \ref{dist-15} relies on the following proposition from \cite[Corollary 9.16]{SSV}.
\begin{proposition}\label{dist-17}
    Let $g$ and $h$ be two entire functions of orders $0<\rho_1,\rho_2<2$ such that $g(0)=h(0)=0$; let $(a_n)_{n\geq 1}$ and $(b_n)_{n\geq 1}$ be two strictly increasing sequences of positive numbers, and let $(\pm a_n)_{n\geq 1}$ and $(\pm b_n)_{n\geq 1}$ be simple roots of $g$ and $h$, respectively. Moreover, assume that $a_n<b_n$ for all $n\geq 1$. Then the function $\phi(t)=h(it) / g(it)$ is the characteristic function of an infinitely divisible probability measure on $\real$.
\end{proposition}
\begin{proof}[Proof of Lemma~\ref{dist-15}]
    We have to show that  $\delta^2(x,0):=-\ln \bigl(p_2^S(x)\big/ p_2^S(0)\bigr)$ is negative definite. Write $\delta^2(x,0)$ in the form  $\delta^2(x,0)=h(ix)\big/ g^3(ix)$, where
\begin{equation*}
    h(z):=\frac{\pi}{2} \Big(\frac{\pi z}{2} \cos (\frac{\pi z}{2})-\sin (\frac{\pi z}{2})  \Big)
    \quad\text{and}\quad
    g(z):=\sin (\frac{\pi z}{2}).
\end{equation*}
    We want to use Proposition \ref{dist-17}. Clearly, both $g$ and $h$ are entire functions, so we need to check only the assumption about their zeroes. The zeroes of $g(z)$ are $(\pm 2k)_{k\geq 0}$. Observe that $h(z)=(\frac{\pi z}{2})^2\, \frac{d}{dz} \left( \sin \bigl(\frac{\pi z}{2}\bigr) \big/ \frac{\pi z}{2} \right)$. Since $\sin \bigl(\frac{\pi z}{2}\bigr)\big/\frac{\pi z}{2}$ is an entire function of order $1$ with exclusively real roots, Laguerre's theorem, see \cite[Theorem 8.5.2]{T39}, shows that the roots of $\frac d{dz} \left( \sin \bigl(\frac{\pi z}{2}\bigr)\big/\frac{\pi z}{2}\right)$ are also real and that they are located between the zeroes of $\sin \bigl(\frac{\pi z}{2}\bigr)\big/\frac{\pi z}{2}$, i.e.\ between $(\pm 2k)_{k>0}$. Moreover, $z=0$ is a root of $\frac d{dz} \left( \sin\bigl(\frac{\pi z}{2}\bigr)\big/\frac{\pi z}{2}\right)$ of multiplicity $1$ because it is located between $-2$ and $2$. Hence, the function $h$ has a root $z=0$ of multiplicity $3$, and otherwise simple roots $b_k, k>0$, located between $(2k)_{k>0}$ on the right half-axis, and $-b_k, k>0$, located  between $(-2k)_{k>0}$ on the left half-axis. Thus we can rewrite $h(z)$ using the Hadamard representation theorem in the form
    $$
        h(z)=\Big(\frac{\pi z}{2}\Big)^3  \prod_{k=1}^\infty \Big(1-\frac{z^2}{b_k^2}\Big).
    $$
    Therefore,
\begin{equation*}
    \delta^2(z,0)
    =\frac{ \prod_{k=1}^\infty \big(1-\frac{z^2}{b_k^2}\big)}{\sin\bigl(\frac{\pi z}{2}\bigr)\big/\frac{\pi z}{2}} \left(\frac{(\frac{\pi z}{2})}{\sin (\frac{\pi z}{2})}\right)^2
    =\phi_1(z)\phi_2(z),
\end{equation*}
where
$$
    \phi_1(z):=\frac{ \prod_{k=1}^\infty \big(1-\frac{z^2}{b_k^2}\big)}{\sin\bigl(\frac{\pi z}{2}\bigr)\big/\frac{\pi z}{2}}
    \quad\text{and}\quad
    \phi_2(z):= \left(\frac{(\frac{\pi z}{2})}{\sin (\frac{\pi z}{2})}\right)^2.
$$
    For $\phi_1$ we can apply Proposition~\ref{dist-17}, since both numerator and denominator are entire functions of order $1$, with simple zeroes satisfying the conditions of the proposition. Hence, $\phi_1(i x )$ is the characteristic function of an infinitely divisible distribution.  As a power of infinitely divisible characteristic functions, $\phi_2(z)$ is again an infinitely divisible characteristic function.
\end{proof}

Let us close this section with an interesting negative result on the class $\Nsf$.

\begin{lemma}[$n$-dimensional version]\label{dist-21}
    Let $p:\rn\to[0,\infty)$ be a rotationally symmetric probability density which is of the form $p(x) = c \,e^{-f(|x|)}$ for some even and increasing function. If $p$ is not the normal distribution and if
    $$
        \lim_{r\to\infty} \frac{f(r)}{r\ln r} = \infty
    $$
    then $p$ cannot be infinitely divisible.
\end{lemma}
\begin{proof}
    Let $X$ be a (non-degenerate) $n$-dimensional random variable with probability density $p(x)$. If $X$ is infinitely divisible and not normally distributed then, by a straightforward modification of \cite[Theorem 26.1]{Sato} (see \cite[Proposition IV.9.8]{SH04} for the exact formula in the one-dimensional setting), we have
    \begin{equation*}
        \limsup_{r\to\infty} \frac{-\ln \Pp(|X|>r)}{r\ln r}=\frac{1}{S}
    \end{equation*}
    where $S = \inf\{R>0\::\: \supp\nu \subset B(0,R)\}$ and $\nu$ is the L\'evy measure in the L\'evy-Khintchine representation of the characteristic exponent of $X$. If $\nu=0$, we set $S=0$. Since $X$ is neither degenerate nor Gaussian, we have $S>0$.

    On the other hand, since $f$ is increasing on $(0,\infty)$, we find for all $r>0$
    $$
        \int_{|x|>r}  e^{-f(|x|)}\,dx
        \leq e^{-\frac{f(r)}{2}} \int_{|x|>r} e^{-\frac{f(|x|)}{2}}\,dx
        = C \, e^{-\frac{f(r)}{2}}.
    $$
    Thus, $\Pp(|X|>r)\leq 2C\, e^{-\frac{f(r)}{2}}$, and
    \begin{equation}
        \limsup_{r\to\infty} \frac{-\ln \Pp(|X|>r)}{r\ln r}
        \geq \limsup_{r\to\infty} \frac{-\ln\big(e^{-\frac{f(r)}{2}}\big)}{r\ln r}
        =\frac 12 \limsup_{r\to\infty } \frac{f(r)}{r\ln r}
        = \infty.
    \end{equation}
    We conclude that $S=0$ and $\supp\nu =\emptyset$ contradicting our assumptions.
\end{proof}

\begin{corollary}\label{dist-23}
     Let $(p_t(x))_{t\geq 0}$ be the transition densities of an $n$-dimensional  symmetric L\'evy process.  If the characteristic exponent $\psi(\xi)$  satisfies $\psi(\xi)\geq f(|\xi|)$ where  $f:(0,\infty)\to(0,\infty)$  is increasing and satisfies $\lim_{r\to\infty} \frac{f(r)}{r\ln r}=\infty$, then none of the densities $p_t(x)$ is of class $\Nsf$.

    \noindent
     In particular,
     \begin{enumerate}[\upshape a)]
     \item
     the transition densities of a rotationally  symmetric $\alpha$-stable process with $1<\alpha<2$ are not of class $\Nsf$;
     \item
     the convolution of a density of class  $\Nsf$ with the normal density is not of class $\Nsf$.
     \end{enumerate}
\end{corollary}
\begin{proof}
    Fix $t>0$ and apply Lemma \ref{dist-21} to the density $p(\xi) := e^{-t\psi(\xi)}/p_t(0)$. Part b) is already contained in \cite{Ru70}, where it is shown that any probability density decreasing faster than $e^{-c|x|^{1+\alpha}}$ as $|x|\to \infty$, cannot be infinitely divisible.
\end{proof}

In Section \ref{table} below we have compiled a list of some known examples of distributions of class $\Nsf$.

\section{The  transition function of certain processes obtained by subordination}\label{sub}

Subordination in the sense of Bochner provides a good tool to get examples and insights. Our first general result for \eqref{intro-e21} to hold, Theorem~\ref{sub-15}, uses subordination. Since the composition of two Bernstein functions is again a Bernstein function, we can understand this theorem as a result for a certain class of subordinate Brownian motions.

\begin{theorem}\label{sub-15}
    Let $(p_t(x))_{t>0}$ be the transition densities of a L\'evy process in $\real$ such that the characteristic functions are of the form $\Ff^{-1} p_t(\xi) = e^{-tf(|\xi|)}$ with some Bernstein function $f$ satisfying $f(0)=0$ and $c_t := \int_0^\infty e^{-tf(r)}\,dr < \infty$ for all $t>0$. Then there exists, for each $t>0$, a complete Bernstein function $g_t$ such that
    $$
        \frac{p_t(x)}{p_t(0)} = e^{-g_t(|x|^2)}.
    $$
\end{theorem}
\begin{proof}
    For every $t>0$ we know by our assumptions that $\rho_t(x):=e^{-tf(x)}$, $x>0$, is completely monotone, hence $\mathcal L\eta_t(x) = e^{-t f(x)}$ for some convolution semigroup $(\eta_t)_{t\geq 0}$ of probability measures on $\real$ with supports in $[0,\infty)$, i.e.\ a subordinator. Since $f$ is a Bernstein function and $|\xi|$ is the characteristic exponent of a Cauchy process, $f(|\xi|)$ is the characteristic exponent of a subordinate Cauchy process. Therefore,
    \begin{align*}
        p_t(x)
        = \frac {1}{\pi} \int_0^\infty \frac s{s^2+ x^2}\,\eta_t(ds)
    \end{align*}
    where we used that $x\mapsto \frac s{\pi\,(s^2+x^2)}$, $s>0$, is the transition density of the Cauchy process.

    Denote by $\eta^*_t$ the pull-back of the measure $s\,\eta_t(ds)$ under the map $J:s\mapsto s^2$. Then
    $$
        p_t(x) = \frac 1{\pi} \int_{[0,\infty)} \frac 1{s+ u}\,\eta^*_t(ds)\bigg|_{u=|x|^2}
        = h_t(u)\big|_{u=|x|^2}
    $$
    where $h_t\in\mathcal S$ is a Stieltjes function. In particular, $h_t$ is an infinitely divisible completely monotone function, cf.\ \cite[Proposition 7.11 and Definition 5.6]{SSV}. Since $h_t(0+) = p_t(0) < \infty$, we can use
    \cite[Lemma 5.7]{SSV} to see that there is some Bernstein function $g_t$ such that
    $$
        \frac{h_t}{h_t(0+)} = e^{-g_t}
        \qquad\text{or}\qquad
        \frac{p_t(x)}{p_t(0)}=e^{-g_t(u)}\bigg|_{u=|x|^2}.
    $$

    Fix $t>0$. Since $h_t\in\mathcal S$ and $h_t(0+) < \infty$, we see that $e^{g_t} = \frac 1{h_t} \in \CBF$ and $\frac 1{h_t(0+)} > 0$. Since we know already that $g_t\in\BF$, we can use \cite[Remark 6.11]{SSV} to conclude that $g_t\in\CBF$ with the representation
    $$
        g_t(s) = \beta_t + \int_0^\infty \frac{s}{s+r}\,\frac{\eta_t(r)}{r}\,dr
    $$
    where $\eta_t : (0,\infty)\to[0,1]$ is measurable.
\end{proof}

\begin{remark}
\textbf{a)}
    Theorem 9.5 in \cite{SSV} shows that the measure $\pi_t$ satisfying $\mathcal L\pi_t = h_t$ is a mixture of exponential distributions.

\medskip\noindent\textbf{b)}
    The converse of Theorem~\ref{sub-15} is, in general, not true. Consider, for example, the function $r\mapsto r^\alpha$, $r>0$, with $\frac{1}{2}<\alpha<1$; this is a complete Bernstein function, but the probability density given by $p(x)= e^{-|x|^{2\alpha}} \big/ \int_\real e^{-|y|^{2\alpha}}\,dy$ is not infinitely divisible, cf.\ Lemma~\ref{dist-21}.

\medskip\noindent\textbf{c)}
    Theorem \ref{sub-15} admits a generalization to dimensions $n \leq 3$. Let $f$ and $\eta_t$ be as in Theorem \ref{sub-15} and assume that $(p_t(x))_{t>0}$ are the transition densities of an $n$-dimensional L\'evy process
    such that $\mathcal F^{-1}p_t(\xi) = e^{-tf(|\xi|)}$.

    As in the proof of Theorem \ref{sub-15} we see that $p_t(x)$ can be written as a mixture of $n$-dimensional Cauchy distributions:
    $$
        p_t(x) = \frac{\Gamma\big(\frac{n+1}2\big)}{\pi^{\frac{n+1}2}} \int_{[0,\infty)} \frac{\lambda}{(\lambda^2+|x|^2)^{\frac{n+1}2}}\,\eta_t(d\lambda).
    $$
    In particular,
    \begin{equation}\label{sec7-eq1}
        \frac{p_t(x)}{p_t(0)}
        = \int_{[0,\infty)} \left(\frac{\lambda^2}{\lambda^2 + |x|^2}\right)^{\frac{n+1}2} \frac{\lambda^{-n}\,\eta_t(d\lambda)}{\int_{[0,\infty)}\lambda^{-n}\,\eta_t(d\lambda)}
        =: \frac{h_t(|x|^2)}{h_t(0+)}.
    \end{equation}
    If $\frac{n+1}{2}\leq 2$, i.e.\ if $n\leq 3$, we know from \cite[Theorem VI.4.9]{SH04} that $h_t(s)/h_t(0+)$, $t>0$, is an infinitely divisible completely monotone function. Now we can follow the argument of the proof of Theorem \ref{sub-15}. Note that Theorem~\ref{sub-15} tells us more: if $n=1$ the function $g_t(x)$ is a complete Bernstein function, which does not follow from \eqref{sec7-eq1}.
\end{remark}

\section{Further examples of processes related to mixtures}\label{proc}

In this section we give a few classes of transition probabilities $(p_t(x))_{t>0}$ where, for each $t>0$, $p_t(x)$ belongs to the class $\Nsf$. These densities are obtained as variance-mean mixtures of the type \eqref{dist-e06} with certain probability measures.

Let $\gauss_t(x)$ be the standard $n$-dimensional Gaussian density
$$
    \gauss_t(x)=\frac{1}{(2\pi t)^\frac{n}{2}}\, e^{-\frac{|x|^2}{2t}},\quad t>0,\; x\in\rn
$$
and observe that
\begin{equation}\label{proc-e42}
    \frac{\gauss_t(x)}{\gauss_t(0)}
    = e^{-\frac{|x|^2}{2t}}
    =\left(\frac{\gauss_1(x)}{\gauss_1(0)}\right)^{\frac{1}{2t}}
    =\Ff \bigl[\gauss_{\frac{1}{2t}}\bigr](x).
\end{equation}
The following proposition should be compared with Theorem~\ref{dist-09}.
\begin{proposition}\label{proc-41}
    Assume that, for each $t>0$, $m_t(ds) $ is an infinitely divisible probability measure on $[0,\infty)$. Denote by $n_t(ds):=\big(\frac{s}{\pi}\big)^{\frac{n}{2}}\,G_\sharp m_t(ds)$ where $G_\sharp m_t$ denotes the pull-back with respect to the map $G: s\mapsto (2s)^{-1}$. If the probability measure $n^*_t(ds):=n_t(ds)/\mathcal{L}[n_t](0)$ is infinitely divisible, i.e.\ if
    \begin{equation*}
        \mathcal{L}[n^*_t](\lambda)=e^{-f_t(\lambda)}
    \end{equation*}
    for some Bernstein function $f_t$, then the variance-mean mixture
    \begin{equation*}
        p_t(x):=\int_0^\infty \gauss_s(x)\,m_t(ds)
    \end{equation*}
    is infinitely divisible, and it is of the form
    \begin{equation*}
        p_t(x) = p_t(0)\,e^{-f_t(|x|^2)}.
    \end{equation*}
\end{proposition}

\begin{proof}
Using \eqref{proc-e42} we see
\begin{align*}
    p_t(x)
    &=\int_{[0,\infty)} \gauss_s(x)\,m_t(ds)\\
    &= \int_{[0,\infty)} \Ff\bigl[\gauss_{\frac{1}{2s}}\bigr](x)\, \gauss_s(0)\, m_t(ds)\\
    &= \int_{[0,\infty)} \Ff\bigl[\gauss_\tau\bigr](x)\,\gauss_{G^{-1}(\tau)}(0)\, G_\sharp m_t(d\tau)\\
    &= \int_{[0,\infty)} e^{-\tau |x|^2} \left(\frac{\tau}{\pi}\right)^{\frac{n}{2}}\, G_\sharp m_t(d\tau)\\
    &= \frac{\mathcal{L}[n_t](0)}{\pi^{\frac{n}{2}}} \, \mathcal{L}\bigl[n^*_t\bigr]\bigl(|x|^2\bigr)\\
    &= p_t(0) e^{-f_t(|x|^2)}.
    \qedhere
\end{align*}
\end{proof}

\begin{example}\label{proc-43}
\textbf{a)}
    Let $(\phi_t)_{t>0}$ be a family of completely monotone functions and set $n_t(s):=s^{-3/2} \phi_t(s)$ and which is normalized to become a probability density $n^*_t(s):= n_t(s)/\mathcal{L}[n_t](0)$. If each $n^*_t$ is an infinitely divisible probability density, then the probability density obtained by mixing
    \begin{equation}\label{proc-e52}
        p_t(x)=\int_{\mathbb{R}} \gauss_s(x)\,n^*_t(s)\,ds
    \end{equation}
    is infinitely divisible, and it is of the form
    \begin{equation*}
        p_t(x)=p_t(0) e^{-f_t(|x|^2)},
    \end{equation*}
    where $f_t\in\CBF$ for each $t>0$. Moreover,
    \begin{equation}\label{proc-e56}
        f_t(s) = \int_0^\infty \frac{s}{s+r}\,\frac{\eta_t(r)}{r}\,dr
    \end{equation}
    where $\eta_t : (0,\infty)\to [0,1]$ is a measurable function.

    \medskip\emph{Indeed:} The fact that $p_t(x)$ is infinitely divisible follows from \eqref{proc-e52} and the infinite divisibility of $n^*_t$ is infinitely divisible. We can rewrite $p_t(x)$ in the following way:
    \begin{align*}
        p_t(x)
        &=\frac{1}{\mathcal{L}[n_t](0)} \int_0^\infty \gauss_s(x)\,s^{-\frac{3}{2}}\, \phi_t\left(\frac{1}{2s}\right) ds\\
        &= \frac{\sqrt{2}}{\sqrt{\pi}\mathcal{L}[n_t](0)} \int_0^\infty e^{-x^2 r} \phi_t(r)\,dr\\
        &= \frac{\sqrt{2}}{\sqrt{\pi}\mathcal{L}[n_t](0)}\,\mathcal{L}[\phi_t](|x|^2).
    \end{align*}
    Since $\phi_t$ is completely monotone, its Laplace transform $\mathcal{L}[\phi_t]$ is a Stieltjes function and our calculation shows that $\infty > p_t(0) = \mathcal{L}[\phi_t](0+)$. Therefore, $1/\mathcal{L}[\phi_t]$ is a complete Bernstein function satisfying $1/\mathcal{L}[\phi_t](0+)>0$. By \cite[Remark 6.11]{SSV} we find that
    $$
        \frac 1{\mathcal{L}[\phi_t](s)} = \text{const}\cdot e^{-f_t(s)}
    $$
    where $f_t(s)$ is a complete Bernstein function of the form \eqref{proc-e56}.

\medskip\noindent\textbf{b)}
    Let $p_t(x)$ be the mixture of a Laplace density with a completely monotone density, i.e.
    \begin{equation}
        p_t(x)
        = \int_0^\infty \frac{s}{2} \, e^{-s|x|}\phi_t(s)\,ds, \quad t>0,
    \end{equation}
    where $\phi_t$ is, for each $t>0$, a completely monotone probability density.  Then
    \begin{equation}\label{proc-e57}
        p_t(x) = c_t \, e^{-g_t(|x|)},
    \end{equation}
    where $c_t>0$ is some constant, depending on $t$, and $g_t$ is  for each $t>0$ a   Bernstein function.

\medskip
    \emph{To see this}, recall that the mixture of Laplace densities is an infinitely divisible probability density, see
    \cite[Theorem IV.10.1]{SH04}. Since $\phi_t(s)$ is completely monotone, $\frac s2\, \phi_t(s)/\int_0^\infty \frac s2\, \phi_t(s)\,ds$ is an infinitely divisible probability density, see \cite[Corollary VI.4.6]{SH04}.  Thus,  $p_t(x)$ is of the form \eqref{proc-e57}, since it is (up to a constant) a Laplace transform  of an infinitely divisible density.
\end{example}

\section{Towards a geometric understanding of transition functions of Feller processes}\label{fel}

In this short section we propose a geometrical approach to understand transition functions of more general processes. This is more of a programme for further studies which should have geometric interests in its own right.

Recall that transition functions for diffusions generated by a second order elliptic differential operator are best understood when using the Riemannian metric associated with the principal part of the generator. Moreover, diffusions generated by subelliptic second order differential operators should be studied in the associated sub-Riemannian geometry. A similar remark applies to diffusions defined on (or defining a) metric measure space. We refer to \cite{Gri1}, \cite{Gri2} and \cite{St} and the references therein.

Let $(X_t)_{t\geq 0}$ be a Feller process, i.e.\ a Markov process such that the semigroup
\begin{equation*}
    T_t u(x) = \Ee^x u(X_t) = \int_\rn u(y)\,p_t(x,dy),\quad u\in C_\infty(\rn)
\end{equation*}
has the Feller property: it preserves the space $C_\infty(\rn)$ of continuous functions vanishing at infinity. We assume that the kernel $p_t(x,dy)$ has a density which we denote, by some abuse of notation, again by $p_t(x,y)$. If the domain $D(A)$ of the generator $A$ of the Feller semigroup $(T_t)_{t\geq 0}$ contains the test functions $C_c^\infty(\rn)$, then $A$ is a pseudo-differential operator with negative definite symbol, i.e.
\begin{equation*}
    Au(x)
    = -q(x,D)u(x)
    = -\int_\rn e^{ix\xi} q(x,\xi)\,\Ff u(\xi)\,d\xi,\quad u\in\mathcal S(\rn),
\end{equation*}
where $q:\rn\times\rn\to\comp$ is a locally bounded function such that $q(x,\cdot)$ is, for every $x\in\rn$, a continuous negative definite function. Throughout we assume that $q(x,\xi)$ is real-valued. We refer to \cite{J2, J3} where many examples of this kind are studied. Moreover, we refer to \cite{J} and, in particular, to \cite{Sch,Sch-Sch} where $q(x,\xi)$ was calculated as
\begin{equation*}
    q(x,\xi) = - \lim_{t\to 0} \frac{\Ee^x\left[e^{i\xi(X_t-x)} - 1\right]}{t}.
\end{equation*}
In \cite{BB1,BB2} B.\ B\"{o}ttcher has proved that for a large class of operators $-q(x,D)$ the symbol of $T_t$
\begin{equation*}
    \lambda_t(x,\xi) = e^{-ix\xi}\,\Ee^x\left[e^{i\xi X_t}\right],\quad t>0,\; x,\xi\in\rn,
\end{equation*}
is asymptotically given by
\begin{equation*}
    \lambda_t(x,\xi) e^{-tq(x,\xi)} + r_0(t;x,\xi)
    \quad\text{as}\quad t\to 0,
\end{equation*}
where $r_0(t;x,\xi)$ tends for $t\to 0$ weakly to zero in the topology of a certain symbol class.

The techniques employed in \cite{J-Sch} will also yield for certain elliptic differential operators $-L(x,D)$ generating the semigroup $(T_t)_{t\geq 0}$ that for the subordinate semigroup $(T_t^f)_{t\geq 0}$ it holds
\begin{equation*}
    \lambda_t(x,\xi)
    =e^{-t f(L(x,\xi))} + \tilde r_0(t;x,\xi)
    \quad\text{as $t\to 0$},
\end{equation*}
and $f(L(x,\xi))$ is, in a certain sense, an approximation of the symbol of $f(L(x,D))$.

Thus we have in several non-trivial cases
\begin{equation*}
    T_t u(x) = \int_\rn e^{ix\xi} e^{-tq(x,\xi)}\,\Ff u(\xi)\,d\xi + \cdots
\end{equation*}
and if $T_t u(x) = \int_\rn p_t(x,y) u(y)\,dy$ we might try as an approximation for $p_t(x,y)$
\begin{equation*}
    p_t(x,y) = \int_\rn e^{i(x-y)\xi} e^{-tq(x,\xi)}\,d\xi + \cdots
\end{equation*}
Now let us assume that for every fixed $x\in\rn$ the continuous negative definite function $q(x,\cdot)$ belongs to $\mcn(\rn)$. In this case we would find
\begin{gather*}
    p_t(x,y) = \int_\rn e^{i(x-y)\xi} e^{-t d^2_{q(x,\cdot)}(\xi,0)} \,d\xi + \cdots
\end{gather*}
and, in particular,
\begin{equation*}
\begin{aligned}
    p_t(x,x)
    &= \int_\rn e^{-t q(x,\xi)}\,d\xi + \cdots\\
    &= \int_\rn e^{-t d^2_{q(x,\cdot)}(\xi,0)} \,d\xi + \cdots
\end{aligned}
\end{equation*}
where $d^2_{q(x,\cdot)}(\xi,\eta) = \sqrt{q(x,\xi-\eta)}$. Moreover, we might think to search for $p_t(x,y)$ an expression of the form
\begin{equation}\label{fel-e22}
    p_t(x,y)
    = p_t(x,x)\, e^{-\delta^2(t,x_0;x,y)}+\cdots,\quad x,y\in B^{d_{q(x_0,\cdot)}}(x_0,\epsilon),
\end{equation}
with a suitable metric $\delta(t,x_0;\cdot,\cdot)$ associated with $q(x_0,\xi)$ in the sense of Section \ref{ptx}.

Thus we propose to switch from $(\rn,d_\psi)$ and $(\rn,\delta_{\psi,t}(\cdot,\cdot)$ to generalizations of Riemannian manifolds: Take $\rn$ with its standard differentiable structure and identify the tangent space $T_x\rn \simeq \rn$. Consider now the families of metrics
\begin{equation*}
    \begin{aligned}
        d_{q(x,\cdot,\cdot)}& : \rn\times\rn \to \real\\
                         &(\xi,\eta)\mapsto \sqrt{q(x,\xi-\eta)}
    \end{aligned}
    \qquad\text{and}\qquad
    \begin{aligned}
        \delta(t,x;\cdot,\cdot)& : \rn\times\rn \to \real\\
                                  &(y,z)\mapsto \delta(t,x;y,z)
    \end{aligned}
\end{equation*}
($\delta(t,x;\cdot,\cdot)$ as proposed in \eqref{fel-e22}) and start to study the corresponding geometric structures. Note that in case that $d_{q(x,\cdot,\cdot)}$ or $\delta(t,x;\cdot,\cdot)$ are related to a continuous negative definite function, our objects to study are manifolds $\rn$ with tangent spaces equipped with a metric such that they allow an isometric embedding into a Hilbert space which is, in general, infinite dimensional.

\newpage\vfill
\begin{small}
\begin{landscape}
\section{A list of probability distributions of class $\Nsf$}\label{table}
\vfill

\begin{center}
\begin{tabular}{|p{37mm}|p{36mm}|p{47mm}|p{47mm}|p{48mm}|}
 \hline
     $X_t$ & $\Ee e^{i\xi \cdot X_t} $& $p_t(x) $& $\psi(\xi)$ & $-\ln \bigl[p_t(x)\big/p_t(0) \bigr]\vphantom{\Bigg]}$\\
      \hline
      {\footnotesize Generalized hyperbolic, $n=1, t=1$. \cite{BKS}}
      & {\footnotesize$\Bigl[\frac{\kappa^2}{\kappa^2+\xi^2}\Bigr]^{\frac{\lambda\vphantom{\int\limits^f}}{2}}
      \frac{K_\lambda\bigl(\delta\sqrt{\kappa^2+\xi^2}\bigr)}{K_\lambda(\delta \kappa)}$}
      & {\footnotesize$\frac{(\frac\kappa\delta)^{\lambda}}{(2\pi)^{\frac {n\vphantom\int}2}K_\lambda(\delta \kappa)}
      \frac{K_{\lambda-n/2} \bigl(\kappa\sqrt{\delta^2+x^2}\bigr)}
      {\bigl(\sqrt{\delta^2+x^2}\,\kappa^{-1}\bigr)^{\frac{n\vphantom\int}2-\lambda}}$}
      & {\footnotesize$-\ln \left[\Bigl[\frac{\kappa^2}{\kappa^2+\xi^2}\Bigr]^{\frac{\lambda}{2}}
      \frac{K_\lambda\bigl(\delta\sqrt{\kappa^2+\xi^2}\bigr)}{K_\lambda(\delta \kappa)}\right]$}
      & {\footnotesize$-\ln \Biggl[\frac{(\frac\kappa\delta)^{\frac{n}{2}+\lambda}K_{\lambda-n/2} \bigl(\kappa \sqrt{\delta^2+x^2}\bigr)}{\bigl(\sqrt{\delta^2+x^2}\,\kappa^{-1} \bigr)^{\frac {n\vphantom\int}2-\lambda}} \Biggr]$}
      \\
\hline
        \footnotesize Normal $N(0,1)$,\newline $F(du)=\delta_t(u)$ & $e^{-t|\xi|^2}$ & $(2\pi t)^{-n/2} e^{-\frac{|x|^2}{2t}}$& $|\xi|^2$ & $\frac{|x|^2}{2t}$ \\
\hline\footnotesize
    Cauchy in $\mathbb{R}$, $n=1$,\newline $\lambda=-1, \alpha=0, \delta=t$.\newline \cite[p.\ 116]{KPST83}
    & $e^{-t|\xi|}$
    & $\frac{1}{\pi} \frac{t}{|x|^2 +t^2}$
    & $|\xi|$
    & $\ln\Big(\frac{\xi^2+t^2}{t^2}\Big)$
    \\
\hline\footnotesize
    Laplace in $\mathbb{R}$, $n=1$,\newline $\lambda=1, \alpha=t, \delta=0$.\newline \cite[p.\ 116]{KPST83}
    & $\frac{t^2}{\xi^2+t^2}$
    & $\frac{t}{2} e^{-t|x|}$
    & $\ln\Big(\frac{\xi^2+t^2}{t^2}\Big)$
    & $t|x|$
    \\
\hline\footnotesize
        Hyperbolic, $t=n=\lambda=1$, \cite[pp.\ 415-6]{J3}
        & $ \frac{\alpha}{K_1(\alpha \delta )} \frac{K_1\big(\delta  \sqrt{\alpha^2+\xi^2}\big)}{\sqrt{\alpha^2+\xi^2}}$
        & $\frac{1}{2\delta K_1(\alpha \delta )} e^{-\alpha \sqrt{\delta^2+x^2}}$
        & $-\ln \left[ \frac{\alpha}{K_1(\alpha \delta )} \frac{K_1\big(\delta  \sqrt{\alpha^2+\xi^2}\big)}{\sqrt{\alpha^2+\xi^2}} \right] $
        & $ \alpha(\sqrt{\delta^2+x^2}-\delta)$
        \\
\hline\footnotesize
        Relativistic Hamiltonian, $n=1, \lambda=-\frac{1}{2}, \delta=t$. \cite[p.\ 182]{J1}
        & $e^{-t \big( \sqrt{\alpha^2+\xi^2}-\alpha\big)}$
        & $\frac{\alpha \delta e^{a \delta} }{\pi} \frac{K_1\big(\alpha \sqrt{\delta^2 +x^2}\big)}{\sqrt{\delta^2 +x^2}}$
        & $\sqrt{\alpha^2+\xi^2}-\alpha$
        & $-\ln \left[ \frac{\delta K_1\big(\alpha \sqrt{\delta^2+x^2}\big)}{K_1(\alpha \delta)} \right]$
        \\
\hline\footnotesize
        Meixner process $C_t$ in $\mathbb{R}$
        \cite[p.\ 312]{PY03}
        & $\big(\frac{1}{\cosh \xi}\big)^t$
        & $ \frac{2^{t-1}}{\pi\Gamma(t)} |\Gamma(\frac{t+i\xi}{2})|^2$
        & $\ln \cosh \xi$
        & $  -\ln \Big| \frac{\Gamma(\frac{t+i\xi}{2})}{\Gamma(\frac{t}{2})}\Big|^2$
        \\[\bigskipamount]
        \footnotesize $t=1$
        &  $\frac{1}{\cosh \xi}$
        &  $\frac{1}{2 \cosh (\frac{\pi x}{2})}$
        &  $\ln \cosh \xi$
        &  $ \ln ( \cosh (\frac{\pi x}{2}))$
        \\[\bigskipamount]
        \footnotesize $t=2$
        & $\big(\frac{1}{\cosh \xi}\big)^2$
        & $\frac{x}{2 \sinh (\frac{\pi x}{2})}$
        & $\ln \cosh \xi$
        &  $\ln \frac{ \sinh (\frac{\pi x}{2})}{\frac{\pi x}{2}} $
        \\
\hline
        \footnotesize \cite[p.\ 312]{PY03}, $t=1$
        & $\frac{\xi}{\sinh \xi}$
        & $\frac{\pi}{4 \cosh^2 (\frac{\pi x}{2})}$
        &$\ln \frac{\sinh \xi}{\xi}$
        & $2 \ln \cosh (\frac{\pi x}{2})$
  \\[\bigskipamount]
        $t=2$
        & $\big(\frac{\xi}{\sinh \xi}\big)^2$
        & $\frac{\frac{\pi}{2} (\frac{\pi x}{2} \coth(\frac{\pi x}{2})-1)}{\sinh^2 (\frac{\pi x}{2})}$
        & $\ln \big(\frac{\sinh \xi}{\xi}\big)$
        & $-\ln \left[\frac{\frac{3\pi x}{2} \coth(\frac{\pi x}{2})-3}{\sinh^2 (\frac{\pi x}{2})}\right]$
        \\
  \hline
\end{tabular}
\end{center}

\vfill\mbox{}

\end{landscape}
\end{small}

\end{document}